\documentclass[review]{elsarticle}

\usepackage{algorithm}
\usepackage{algorithmic}
\usepackage{graphicx}
\usepackage{epstopdf}
\usepackage{diagbox}
\usepackage{amsthm,amsfonts,amssymb}
\usepackage{multirow}

\usepackage{lipsum}



\usepackage{amsmath}
\allowdisplaybreaks[4]
\usepackage{amsmath}
\usepackage{amssymb}
\usepackage{url}
\usepackage{amsthm,amsfonts,amssymb}

\newtheorem{theorem}{Theorem}[section] 
\newtheorem{definition}[theorem]{Definition} 
\newtheorem{lemma}{Lemma}
\newtheorem{corollary}[theorem]{Corollary}
\newtheorem{example}{Example}[section]

\journal{Journal of \LaTeX\ Templates}

\begin{document}

\begin{frontmatter}



    \title{Dual Complex Adjoint Matrix: Applications in Dual Quaternion Research \tnoteref{mytitlenote}}
    \tnotetext[mytitlenote]{The author Liping Zhang  is supported by the National Natural Science Foundation of China (Grant No. 12171271).}


    \author{Yongjun Chen}
    \author{Liping Zhang\corref{mycorrespondingauthor}}
    \cortext[mycorrespondingauthor]{Corresponding author. \emph{Email address}:lipingzhang@mail.tsinghua.edu.cn}
    \address{Department of Mathematical Sciences, Tsinghua University, Beijing 100084, China}

\begin{abstract}
Dual quaternions and dual quaternion matrices have garnered widespread applications in robotic research, and its spectral theory has been extensively studied in recent years. This paper introduces the novel concept of the dual complex adjoint matrices for dual quaternion matrices. We delve into exploring the properties of this matrix, utilizing it to study eigenvalues of dual quaternion matrices and defining the concept of standard right eigenvalues. Notably, we leverage the properties of the dual complex adjoint matrix to devise a direct solution to the Hand-Eye calibration problem. Additionally, we apply this matrix to solve dual quaternion linear equations systems, thereby advancing the Rayleigh quotient iteration method for computing eigenvalues of dual quaternion Hermitian matrices, enhancing its computational efficiency. Numerical experiments have validated the correctness of our proposed method in solving the Hand-Eye calibration problem and demonstrated the effectiveness in improving the Rayleigh quotient iteration method, underscoring the promising potential of dual complex adjoint matrices in tackling dual quaternion-related challenges.
\end{abstract}




\end{frontmatter}



\section{Introduction}\label{Introduction}
Dual quaternion, first introduced by Clifford in 1873 \cite{biq}, have wide applications in 3D computer graphics, robotics control and computer vision \cite{c1,Cheng2016,Daniilidis1999,c4,c5}. Dual quaternion matrices occupy a pivotal role in robotic research, notably in tackling problems like simultaneous localization and mapping (SLAM) \cite{qc1,c1,Cadena2016,Grisetti2010} and kinematic modeling \cite{Qi2023}.

The Hand-Eye calibration problem stands as a pivotal issue in the realm of robot calibration, with extensive applications spanning aerospace, medical, automotive, and industrial sectors \cite{qc6,qc7}. This problem is to determine the homogeneous matrix between the camera frame and the end-efector frame or between a robot base and the world coordinate system. The $AX=XB$ and $AX=YB$ mathematical model are two main hand-eye calibration model \cite{qc9,qc10}. Recognizing that robotic poses can be elegantly represented using unit dual quaternions, Hand-Eye calibration problem can be reformulated within the framework of dual quaternions \cite{Daniilidis1999}.

The spectral theory of dual quaternion matrices has been explored in \cite{d2,Li2023}, from which we know if $\hat{\lambda}$ is a right eigenvalue of a dual quaternion matrix, then $\hat{q}^{-1}\hat{\lambda}\hat{q}$ is also a right eigenvalue of that matrix for any invertible
dual quaternion $\hat{q}$. This leads to an equivalent class $[\lambda]=\{\hat{q}^{-1}\hat{\lambda}\hat{q}|\hat{q} \in \hat{\mathbb Q}~is~invertible\}$ of the right eigenvalue $\lambda$. A natural question is whether this equivalence class has a representative element. Subsequently,
paper \cite{d3} delved into the minimax principle pertaining to the eigenvalues of dual quaternion Hermitian matrices and paper \cite{Ling2023} studied the Hoffman-Wielandt type inequality for dual quaternion Hermitian matrices by employing von Neumann type trace inequality. Furthermore, paper \cite{Li2023} emphasized the pivotal role of the eigenvalue theory of dual quaternion Hermitian matrices in analyzing the stability of multi-agent formation control systems.

In recent research \cite{Duan2023}, the Rayleigh quotient iteration method has emerged as an effective tool for computing the extreme eigenvalues of dual quaternion Hermitian matrices. An extreme eigenvalue is defined as one whose standard part's magnitude surpasses or falls below that of all other eigenvalues. Notably, this method boasts locally cubic convergence, underscoring its efficiency. The eigenvalue obtained through this method is contingent upon the initial iteration point, emphasizing the significance of selecting an appropriate starting point. Furthermore, each iteration necessitates solving a dual quaternion linear equations system, which can become computationally intensive for matrices of substantial dimensions.

The complex adjoint matrix of quaternion matrix is a valuable tool in exploring the eigenvalue theory of quaternion matrices \cite{Zhang1997}, effectively transforming eigenvalue problems of quaternion matrix into one of complex adjoint matrix. Notably, if $\tilde{\lambda}$ is a right eigenvalue of a quaternion matrix, then $\tilde{q}^{-1}\tilde{\lambda}\tilde{q}$ is also a right eigenvalue of that matrix for any invertible quaternion $\tilde{q}$. It has been established that an $n\times n$ quaternion matrix possesses exactly $n$ complex eigenvalues, referred to as its standard eigenvalues. Inspired by the complex adjoint matrix, we endeavor to define an analogous matrix which we called dual complex adjoint matrix for dual quaternion matrices, aiming to preserve the advantageous properties of the complex adjoint matrix. Our objective is to investigate the properties of the dual complex adjoint matrix and leverage it to delve into the eigenvalue theory of dual quaternion matrices, as well as address the Hand-Eye calibration problem and dual quaternion linear equations systems.

The structure of this paper is organized as follows. In Section \ref{Preliminaries}, we present an overview of the fundamental definitions and results pertaining to dual quaternions and dual quaternion matrices. In Section \ref{section3}, we introduce the dual complex adjoint matrix and highlight some useful properties. Utilizing this matrix, we delve into the study of eigenvalues of dual quaternion matrices, the Hand-eye calibration problem and the dual quaternion linear equations systems. This leads to the  improvement of the Rayleigh quotient iteration method based on the dual complex adjoint matrix, for solving the eigenvalues of dual quaternion Hermitian matrices. In Section \ref{section4}, we present an example to solve Hand-Eye calibration problem and employ the modified Rayleigh quotient iteration method to address the eigenvalue problem of the Laplacian matrix of graphs in the multi-agent formation control problem. Some final remarks are drawn in Section \ref{section5}.

\section{Preliminary}\label{Preliminaries}

In this section, we introduce some preliminary knowledge about the dual number, dual complex number, quaternion, dual quaternion, dual quaternion matrix, and the eigenvalue theory of dual quaternion matrix. 

\subsection{dual quaternion}
Denote $\mathbb{R}$, $\mathbb{C}$, $\mathbb{D}$, $\mathbb{DC}$, $\mathbb{Q}$, $\mathbb{U}$, $\hat{\mathbb{Q}}$ and $\hat{\mathbb{U}}$ as the set of real numbers, complex numbers, dual numbers, dual complex numbers, quaternions, unit quaternions, dual quaternions, and unit dual quaternions, respectively.

The symbol $\varepsilon$ is denoted as the infinitesimal unit, satisfies $\varepsilon\neq 0$ and $\varepsilon^{2}=0$.
$\varepsilon$ is commutative with complex numbers and quaternions.

\subsubsection{Dual number and dual complex number}

\begin{definition}\label{def1}
A {\bf dual complex number} $\hat{a}=a_{st}+a_{\mathcal I}\varepsilon\in \mathbb{DC}$ has standard part
$a_{st}\in \mathbb{C}$ and dual part $a_{\mathcal I}\in \mathbb{C}$. We say that $\hat{a}$ is {\bf appreciable} if $a_{st}\neq 0$. If $a_{st},a_{\mathcal I}\in \mathbb{R}$, then $\hat{a}$ is called a {\bf dual number}.
\end{definition}

The following definition lists some operators about dual complex numbers, which can be found in \cite{Qi2022}..

\begin{definition}\label{def2}
Let $\hat{a}=a_{st}+a_{\mathcal I}\varepsilon$ and $\hat{b}=b_{st}+b_{\mathcal I}\varepsilon$ be any two {\bf dual complex numbers}. The conjugate,  absolute value of $\hat{a}$, and the addition, multiplication, division between $\hat{a}$ and $\hat{b}$ are defined as follows.
\begin{itemize}
\item[{\rm (i)}] The {\bf conjugate} of $\hat{a}$ is
$\hat{a}^{\ast }=a_{st}^{\ast}+a_{\mathcal I}^{\ast }\varepsilon,$ where $a_{st}^{\ast}$ and $a_{\mathcal I}^{\ast}$ are conjugates of the complex numbers $a_{st}$ and $a_{\mathcal I}$, respectively.

\item[{\rm (ii)}] The {\bf addition} and {\bf multiplication} of $\hat{a}$ and $\hat{b}$ are
$$ \hat{a}+\hat{b}=\hat{b}+\hat{a}=\left (a_{st}+b_{st} \right )+\left (a_{\mathcal I}+b_{\mathcal I} \right )\varepsilon,$$
and
$$\hat{a}\hat{b}=\hat{b}\hat{a}=a_{st}b_{st} +\left (a_{st}b_{\mathcal I}+a_{\mathcal I}b_{st} \right )\varepsilon.$$

\item[{\rm (iii)}] When $b_{st}\ne 0$ or $a_{st}=b_{st}=0$, we can define the division operation of dual numbers as
$$\frac{\hat{a}}{\hat{b}}=\begin{cases}
\dfrac{ a_{st}}{b_{st}}+\left (\dfrac{ a_{\mathcal I}}{b_{st}}-\dfrac{ a_{st}b_{\mathcal I}}{b_{st}b_{st}} \right )\varepsilon, & {\rm if} \quad b_{st}\neq 0, \\
\dfrac{ a_{\mathcal I}}{b_{\mathcal I}}+c\varepsilon, & {\rm if} \quad a_{st}=b_{st}=0,
\end{cases}
$$
where $c$ is an arbitrary complex number.

\item[{\rm (iv)}] The {\bf absolute value}  of $\hat{a}$ is
$$\left | \hat{a}\right |=\begin{cases}
\left | a_{st}\right |+{\rm sgn}( a_{st}) a_{\mathcal I}\varepsilon, & {\rm if} \quad a_{st}\neq 0, \\
\left | a_{\mathcal I} \right |\varepsilon, & \text{\rm otherwise.}
\end{cases}$$

\end{itemize}
\end{definition}

The order operation for dual numbers was defined in \cite{Cui2023}.

\begin{definition}\label{def3}
Let $\hat{a}=a_{st}+a_{\mathcal I}\varepsilon$, $\hat{b}=b_{st}+b_{\mathcal I}\varepsilon$ be any two {\bf dual numbers}. We give the order operation between $\hat{a}$ and $\hat{b}$. We say that $\hat{a} > \hat{b}$ if
$$\text{$a_{st} > b_{st}$ ~~~{\rm or}~~~ $a_{st}=b_{st}$ {\rm and} $a_{\mathcal I} > b_{\mathcal I}$}.$$
\end{definition}

\subsubsection{quaternion}

\begin{definition}\label{def4}
A {\bf quaternion} can be represented as $\tilde{q}=q_{0}+q_{1} i+q_{2} j+q_{3} k$, where $q_{0},q_{1},q_{2},q_{3}$ are real numbers, 
$i, j, k$ satisfy $ij=k$, $jk=i$,$ki=j$ and $i^2=j^2= k^2=-1$. 
$\tilde{q}$ can also be represented as $\tilde{q}=\left [q_{0}, q_{1},q_{2},q_{3}\right ]$, which is a real four-dimensional vector. We can rewrite $\tilde{q}=[ q_{0},\vec{q} ]$, where $\vec{q}$ is a real three-dimensional vector \cite{Cheng2016,Daniilidis1999,Wei2013}.
\end{definition}

Denote $\tilde{1}=\left [1,0,0,0\right ]\in \mathbb{Q}$ and $\tilde{0}=\left [0,0,0,0\right ]\in \mathbb{Q}$ as the identity element and zero element of the quaternion set $\mathbb {Q}$.

The following definition lists some operators about quaternions.
\begin{definition}\label{def5}
Let $\tilde{p}=\left [p_{0}, \vec{p}\right ]$ and $\tilde{q}=\left [q_{0}, \vec{q}\right ]$  be any two {\bf quaternions}. The conjugate, magnitude, and inverse of $\tilde{p}$, and the addition, multiplication between $\tilde{q}$ and $\tilde{p}$ are defined as follows.
\begin{itemize}
\item[{\rm (i)}] The conjugate of $\tilde{p}$ is $\tilde{p}^{\ast}=\left [p_{0}, -\vec{p}\right ]$.
\item[{\rm (ii)}] The magnitude of $\tilde{p}$ is 
$\left |\tilde{p}\right|=\sqrt{p_{0}^{2}+\left \| \vec{p} \right \|_2^2}$.
\item[{\rm (iii)}] The addition and multiplicity of $\tilde{q}$ and $\tilde{p}$ are 
\begin{equation}
\tilde{p}+\tilde{q}=\tilde{q}+\tilde{p}=\left [p_{0}+q_{0},\vec{p}+\vec{q}\right ]   
\end{equation}
and 
\begin{equation}
\tilde{p}\tilde{q}=\left [p_{0}q_{0}-\vec{p}\cdot \vec{q}, p_{0}\vec{q}+q_{0}\vec{p}+\vec{p}\times \vec{q}\right ].
\end{equation}
\item[{\rm (iv)}]
If $\tilde{q}\tilde{p}=\tilde{p}\tilde{q}=\tilde{1}$, then we call the quaternion $\tilde{p}$ is invertible and the inverse of $\tilde{p}$ is $\tilde{p}^{-1}=\tilde{q}$.
\end{itemize}
\end{definition}

\begin{definition}\label{def6}
We say that $\tilde{p}\in \mathbb{Q}$ is a {\bf unit quaternion} if $\left |\tilde{p} \right |=1$.
Clearly, if $\tilde{p}$ and $\tilde{q}$ are unit quaternions, i.e., $\tilde{p},\tilde{q}\in \mathbb{U}$, then $\tilde{p}\tilde{q}\in \mathbb{U}$. Furthermore, we have $\tilde{p}^{\ast }\tilde{p}=\tilde{p}\tilde{p}^{\ast }=\tilde{1}$, i.e., $\tilde{p}$ is invertible and $\tilde{p}^{-1}=\tilde{p}^{\ast }$. Generally, if quaternion $\tilde{p}\ne \tilde{0}$, we have  $\frac{\tilde{p}^{\ast }}{|\tilde{p}|^2}\tilde{p}=\tilde{p}\frac{\tilde{p}^{\ast }}{|\tilde{p}|^2}=\tilde{1}$, then $\tilde{p}^{-1}=\frac{\tilde{p}^{\ast }}{|\tilde{p}|^2}$.
\end{definition}

\subsubsection{dual quaternion}
\begin{definition}\label{defss}
A {\bf dual quaternion} $\hat{p}=\tilde{p}_{st}+\tilde{p}_{\mathcal I}\varepsilon\in \hat{\mathbb{Q}}$ has standard part
$\tilde{p}_{st} \in \mathbb{Q}$ and dual part $\tilde{p}_{\mathcal I} \in \mathbb{Q}$. We say that $\hat{q}$ is {\bf appreciable} if $\tilde{p}_{st}\neq \tilde{0}$. 
\end{definition}

Denote $\hat{1}=\tilde{1}+\tilde{0}\varepsilon\in \hat{\mathbb{Q}}$
and $\tilde{0}=\tilde{0}+\tilde{0}\varepsilon\in \hat{\mathbb{Q}}$
as the identity element and zero element of the dual quaternion set $\hat{\mathbb {Q}}$.

The following definition lists some operators about dual quaternions.

\begin{definition}\label{def8}
Let $\hat{p}=\tilde{p}_{st}+\tilde{p}_{\mathcal I}\varepsilon$ and $\hat{q}=\tilde{q}_{st}+\tilde{q}_{\mathcal I}\varepsilon$ be {\bf dual quaternions}. The conjugation, absolute value and magnitude of $\hat{p}$ and the addition, multiplicity and division of $\hat{p}$ and $\hat{q}$ are defined as follows.
\begin{itemize}
\item[{\rm (i)}] The conjugate of $\hat{p}$ is 
$\hat{p}^{\ast }=\tilde{p}_{st}^{\ast}+\tilde{p}_{\mathcal I}^{\ast }\varepsilon$, where $\tilde{p}_{st}^{\ast}$ and $\tilde{p}_{\mathcal I}^{\ast}$ are the conjugate of quaternions.
\item[{\rm (ii)}] The absolute value \cite{Qi2022} of $\hat{p}$ is 
\begin{equation}
\left | \hat{p}\right |=\begin{cases}
\left | \tilde{p}_{st}\right |+\frac{sc(\tilde{p}_{st}^{\ast }\tilde{p}_{\mathcal I})}{\left |\tilde{p}_{st} \right |}\varepsilon, & \text{if} \quad \tilde{p}_{st}\neq \tilde{0}, \\
\left | \tilde{p}_{\mathcal I} \right |\varepsilon, & \text{if} \quad \tilde{p}_{st}= \tilde{0},
\end{cases}
\end{equation}
where $sc(\tilde{p})=\frac{1}{2}(\tilde{p}+\tilde{p}^{\ast})$.
We say that $\hat{p}$ is a {\bf unit dual quaternion} if $\left |\hat{p} \right |=1$.
\item[{\rm (iii)}] The magnitude of $\hat{p}$ is
\begin{equation}
\left | \hat{p}\right |_2=\sqrt{|\tilde{p}_{st}|^2+|\tilde{p}_{\mathcal I}|^2}.
\end{equation}
\item[{\rm (iv)}] The addition and multiplicity of $\hat{p}$ and $\hat{q}$ are 
\begin{equation}
\hat{p}+\hat{q}=\hat{q}+\hat{p}=\left (\tilde{p}_{st}+\tilde{q}_{st} \right )+\left (\tilde{p}_{\mathcal I}+\tilde{q}_{\mathcal I} \right )\varepsilon,
\end{equation}
and 
\begin{equation}
\hat{p}\hat{q}=\tilde{p}_{st}\tilde{q}_{st} +\left (\tilde{p}_{st}\tilde{q}_{\mathcal I}+\tilde{p}_{\mathcal I}\tilde{q}_{st} \right )\varepsilon.
\end{equation}
\item[{\rm (v)}] If $\tilde{q}_{st}\ne \tilde{0}$ or  $\tilde{q}_{st}=\tilde{q}_{st}=\tilde{0}$, we can define the division of $\hat{p}$ and $\hat{q}$ as
\begin{equation}
\frac{ \tilde{p}_{st}+ \tilde{p}_{\mathcal I}\varepsilon}{\tilde{q}_{st}+ \tilde{q}_{\mathcal I}\varepsilon}=\begin{cases}
\dfrac{ \tilde{p}_{st}}{\tilde{q}_{st}}+\left (\dfrac{ \tilde{p}_{\mathcal I}}{\tilde{q}_{st}}-\dfrac{ \tilde{p}_{st}\tilde{q}_{\mathcal I}}{\tilde{q}_{st}\tilde{q}_{st}} \right )\varepsilon, & \text{if} \quad \tilde{q}_{st}\neq \tilde{0}, \\
\dfrac{ \tilde{p}_{\mathcal I}}{\tilde{q}_{\mathcal I}}+\tilde{c}\varepsilon, & \text{if} \quad \tilde{p}_{st}=\tilde{q}_{st}=\tilde{0},
\end{cases}
\end{equation}
where $\tilde{c}$ is an arbitrary quaternion.
\item[{\rm (vi)}] If $\hat{q}\hat{p}=\hat{p}\hat{q}=\hat{1}$, then $\hat{p}$ is called invertible and the inverse of $\hat{p}$ is $\hat{p}^{-1}=\hat{q}$.
If $\hat{p}\in \hat{\mathbb{U}}$, then $\hat{p}^{\ast }\hat{p}=\hat{p}\hat{p}^{\ast }=\hat{1}$, i.e., $\hat{p}$ is invertible and $\hat{p}^{-1}=\hat{p}^{\ast}$. Generally, if $\hat{p}$ is appreciable, we have $\hat{p}^{-1}=\tilde{p}^{-1}_{st}-\tilde{p}^{-1}_{st}\tilde{p}_{\mathcal I}\tilde{p}^{-1}_{st}\varepsilon$.
\end{itemize}
\end{definition}


\subsection{Dual quaternion matrix and eigenvalues of dual quaternion matrix}
The sets of dual number matrix, dual complex matrix, quaternion matrix, and dual quaternion matrix with dimension $n\times m$  are denoted as
$\mathbb{D}^{n\times m}$, $\mathbb{DC}^{n\times m}$, $\mathbb{Q}^{n\times m}$, and $\hat{\mathbb{Q}}^{n\times m}$.
Denote $O^{n\times m}$, $\hat{O}^{n\times m}$, $\tilde{\mathbf{O}}^{n\times m}$ and  $\hat{\mathbf{O}}^{n\times m}$ as the zero element of the set of complex matrices, dual complex matrices, quaternion matrices and dual quaternion matrices with dimension $n\times m$, respectively.
Denote $\hat{I}_n$, $\tilde{\mathbf{I}}_n$ and $\hat{\mathbf{I}}_n$ as the identity element of the set of dual complex matrices, quaternion matrices and dual quaternion matrices with dimension $n\times n$, respectively.

A quaternion matrix $\tilde{\mathbf Q}\in \mathbb Q^{m\times n}$ can be 
expressed as $\tilde{\mathbf Q}= Q_1+ Q_2 i+ Q_3 
j+ Q_4 k\in \mathbb Q^{m\times n}$, where $ Q_1, Q_2, Q_3, Q_4\in \mathbb R^{m\times n}$.
Denote $P_1= Q_1+Q_2 i$ and $P_2= Q_3+ Q_4 i$, then we can rewrite $\tilde{\mathbf Q}$ as 
$\tilde{\mathbf Q}= P_1+ P_2 j$.
The $F$-norm of quaternion matrix $\tilde{\mathbf{Q}}=(\tilde{q}_{ij}) \in \tilde{\mathbb{Q}}^{m\times n}$ is defined as
$\| \tilde{\mathbf{Q}}\| _F=\sqrt{\sum_{ij}^{} |\tilde{q}_{ij}|^2}$,
and the magnitude of the quaternion vector $\tilde{\mathbf{x}}=(\tilde{x}_{i}) \in \tilde{\mathbb{Q}}^{n\times 1}$ is defined as
$\left \| \tilde{\mathbf{x}} \right \|=\sqrt{\sum_{i=1}^{n}|\tilde{x}_{i}|^2}.$

A dual complex matrix $\hat{P}=P_{st}+P_{\mathcal I}\varepsilon \in \mathbb{DC}^{n\times m}$ has standard part $P_{st}\in \mathbb C^{n\times m}$ and dual part $P_{\mathcal I}\in \mathbb C^{n\times m}$. If $P_{st}\neq O$, then $\hat{P}$ is called appreciable. The transpose and the conjugate of $\hat{P}=(\hat{p}_{ij})$ are $\hat{P}^{T}=(\hat{p}_{ji})$ and $\hat{P}^{\ast}=(\hat{p}_{ji}^{\ast })$, respectively. If $\hat{P}^{\ast}=\hat{P}$, then $\hat{P}$ is a dual complex Hermitian matrix. A dual complex matrix $\hat{U}\in \mathbb{DC}^{n\times n}$ is a unitary matrix, if $\hat{U}^\ast\hat{U}=\hat{U}\hat{U}^\ast=\hat{I}_n$.

\subsubsection{Dual quaternion matrix}
A dual quaternion matrix $\hat{\mathbf Q}=\tilde{\mathbf Q}_{st}+\tilde{\mathbf Q}_{\mathcal I}\varepsilon \in \hat{\mathbb{Q}}^{m\times n}$ has standard part  $\tilde{\mathbf Q}_{st}\in \mathbb Q^{m\times n}$ and dual part $\tilde{\mathbf Q}_{\mathcal I}\in \mathbb Q^{m\times n}$. If $\tilde{\mathbf Q}_{st}\neq \tilde{\mathbf O}$, then $\hat{\mathbf Q}$ is called appreciable. The transpose and conjugate of $\hat{\mathbf{Q}}=(\hat{q}_{ij})$ are 
$\hat{\mathbf{Q}}^{T}=(\hat{q}_{ji})$ and $\hat{\mathbf{Q}}^{\ast}=(\hat{q}_{ji}^{\ast })$, respectively. $\hat{\mathbf{U}}\in \hat{\mathbb{Q}}^{n\times n}$ is a unitary dual quaternion matrix, if $\hat{\mathbf{U}}^\ast\hat{\mathbf{U}}=\hat{\mathbf{U}}\hat{\mathbf{U}}^\ast=\hat{\mathbf{I}}_n$. The set of unitary dual quaternion matrix with dimension $n$ is denoted as $\hat{\mathbb{U}}^n_2$. If $\hat{\mathbf{Q}}^{\ast}=\hat{\mathbf{Q}}$, then $\hat{\mathbf{Q}}$ is called a dual quaternion Hermitian matrix. The set of dual quaternion Hermitian matrix with dimension $n$ is denoted as $\hat{\mathbb{H}}^n$.

The following definition gives $2$-norm, $2^{R}$-norm for dual quaternion vectors and $F$-norm, $F^{R}$-norm for dual quaternion matrices. See \cite{Cui2023,Ling2023,Qi2022}.
\begin{definition}\label{def10}
Let $\hat{\mathbf{x}}=\tilde{\mathbf{x}}_{st}+\tilde{\mathbf{x}}_{\mathcal{I}}\varepsilon=(\hat{x}_i) \in \hat{\mathbb{Q}}^{n\times 1}$ and $\hat{\mathbf{Q}}=\tilde{\mathbf{Q}}_{st}+\tilde{\mathbf{Q}}_{\mathcal{I}}\varepsilon \in \hat{\mathbb{Q}}^{m\times n}$.
The $2$-norm and $2^{R}$-norm for dual quaternion vectors are respectively defined by
\begin{equation}
\| \hat{\mathbf x}\|_{2}=\begin{cases}
\sqrt{\sum_{i=1}^{n}|\hat{x}_{i} |^{2}}, & \text{if} \quad \tilde{\mathbf x }_{st}\neq \tilde{\mathbf O }, \\
\|\tilde{\mathbf{x}}_{\mathcal{I}}\|\varepsilon, & \text{if} \quad \tilde{\mathbf x }_{st}=\tilde{\mathbf O },
\end{cases}  
\end{equation}
and 
\begin{equation}
\| \hat{\mathbf x}\|_{2^{R}}=\sqrt{\| \tilde{\mathbf x}_{st}  \|^{2}+\| \tilde{\mathbf x }_{\mathcal I} \|^{2}}.
\end{equation}
The set of $n\times 1$ dual quaternion vectors with unit $2$-norm is denoted as $\hat{\mathbb{Q}}_2^{n\times 1}$.

The $F$-norm and $F^{R}$-norm for dual quaternion matrices are defined by
\begin{equation}
\| \hat{\mathbf Q} \|_{F}=\begin{cases}
\| \tilde{\mathbf Q}_{st} \|_{F}+\frac{sc (tr(\tilde{\mathbf Q}_{st}^{\ast }\tilde{\mathbf Q}_{\mathcal I}))}{\|  \tilde{\mathbf Q}_{st} \|_{F}}\varepsilon, & \text{if} \quad \tilde{\mathbf Q}_{st}\neq \tilde{\mathbf O }, \\
 \| \tilde{\mathbf Q}_{\mathcal I} \|_{F}\varepsilon, & \text{if} \quad \tilde{\mathbf Q}_{st}=\tilde{\mathbf O },
\end{cases} 
\end{equation}
and 
\begin{equation}
 \| \hat{\mathbf Q} \|_{F^{R}}=\sqrt{ \| \tilde{\mathbf Q}_{st} \|_{F}^{2}+\| \tilde{\mathbf Q}_{\mathcal I}\|_{F}^{2}}.
\end{equation}
\end{definition}
\subsubsection{Eigenvalues of dual quaternion matrices}
The following definition introduces eigenvalues and eigenvectors of dual quaternion matrices \cite{Li2023}.
\begin{definition}\label{def12}
Let $\hat{\mathbf Q}\in \hat{\mathbb{Q}}^{n\times n}$.
If there exist $\hat{\lambda} \in \hat{\mathbb{Q}}$ and $\hat{\mathbf x}\in \hat{\mathbb{Q}}^{n\times 1}$, where $\hat{\mathbf x}$ is appreciable, such that
\begin{equation}
    \hat{\mathbf Q}\hat{\mathbf x}=\hat{\mathbf x}\hat{\lambda},
\end{equation}
then we call $\hat{\lambda}$ is a {\bf right eigenvalue} of $\hat{\mathbf Q}$ with $\hat{\mathbf x}$ as an associated {\bf right eigenvector}.

If there exist $\hat{\lambda} \in \hat{\mathbb{Q}}$ and $\hat{\mathbf x}\in \hat{\mathbb{Q}}^{n\times 1}$, where $\hat{\mathbf x}$ is appreciable, such that
\begin{equation}
    \hat{\mathbf Q}\hat{\mathbf x}=\hat{\lambda}\hat{\mathbf x},
\end{equation}
then we call $\hat{\lambda}$ is a {\bf left eigenvalue} of $\hat{\mathbf Q}$ with $\hat{\mathbf x}$ as an associated {\bf left eigenvector}.

Since a dual number is commutative with a dual quaternion vector, then if $\hat{\lambda}$ is a dual number and a right eigenvalue of $\hat{\mathbf{Q}}$, it is also a left eigenvalue of $\hat{\mathbf{Q}}$.  In this case, we simply call $\hat{\lambda}$ an  {\bf eigenvalue} of $\hat{\mathbf{Q}}$ with $\hat{\mathbf{x}}$ as an associated {\bf eigenvector}.
\end{definition}

A dual quaternion Hermitian matrix $\hat{\mathbf{Q}}\in\hat{\mathbb{H}}^n$ has exactly $n$ eigenvalues, which are all dual numbers. Similar to the case of Hermitian matrix, we have unitary decomposition of $\hat{\mathbf{Q}}$, namely, there exist a unitary dual quaternion matrix $\hat{\mathbf{U}}\in \hat{\mathbb{U}}^{n}_2$ and a diagonal dual number matrix $\hat{\Sigma} \in \mathbb D^{n\times n}$ such that  $\hat{\mathbf{Q}}=\hat{\mathbf{U}}^*\hat{\Sigma}\hat{\mathbf{U}}$ \cite{Li2023}.

Duan et al \cite{Duan2023} proposed the Rayleigh quotient iteration method for calculating the extreme eigenvalues of dual quaternion Hermitian matrices. Given any $\hat{\mathbf{Q}}\in\hat{\mathbb{H}}^n$ and initial iterative vector $\hat{\mathbf{v}}^{(0)}\in \hat{\mathbb{Q}}_2^{n\times 1}$, the $k$-th iteration of Rayleigh quotient iteration method is
\begin{equation}
\hat{\lambda}^{(k-1)}=(\hat{\mathbf{v}}^{(k-1)})^*\hat{\mathbf{Q}}\hat{\mathbf{v}}^{(k-1)},~\text{Solve}~ (\hat{\mathbf{Q}}-\hat{\lambda}^{(k-1)}\hat{\mathbf{I}})\hat{\mathbf{u}}^{(k)}=
\hat{\mathbf{v}}^{(k-1)},\hat{\mathbf{v}}^{(k)}=
\frac{\hat{\mathbf{u}}^{(k)}}{\|\hat{\mathbf{u}}^{(k)}\|_2}.
\end{equation}

$\hat{\lambda}_1$ is called the extreme eigenvalue of $\hat{\mathbf Q}$, if 
\begin{equation}
|\lambda_{1,\mathrm{st}}|>|\lambda_{2,\mathrm{st}}|\geq\cdots\geq|\lambda_{n,\mathrm{st}}|\geq0
\end{equation} or 
\begin{equation}
0 \le |\lambda_{1,\mathrm{st}}|<|\lambda_{2,\mathrm{st}}|\le \cdots\le |\lambda_{n,\mathrm{st}}|,
\end{equation}
where $\{\hat{\lambda}_i=\lambda_{i,st }+ \lambda_{i, \mathcal I}\varepsilon\}_{i=1}^n$ are eigenvalues of $\hat{\mathbf Q}$.

An important step in the Rayleigh quotient iteration method is to solve a dual quaternion linear equations system. Suppose that $\hat{\mathbf{Q}}=\hat{Q}_0+\hat{Q}_1i+\hat{Q}_2j+\hat{Q}_3k\in \hat{\mathbb{Q}}^{m\times n}$, where $\hat{Q}_t\in \mathbb{D}^{m\times n},t=0,1,2,3$, then the dual representation of $\hat{\mathbf{Q}}$ is 
\begin{equation}
\hat{\mathbf{Q}}^D=\begin{bmatrix}
\hat{Q}_0  & \hat{Q}_1 & \hat{Q}_2 & \hat{Q}_3\\
-\hat{Q}_1  & \hat{Q}_0 &- \hat{Q}_3 & \hat{Q}_2\\
-\hat{Q}_2  & \hat{Q}_3 & \hat{Q}_0 &- \hat{Q}_1\\
-\hat{Q}_3  & -\hat{Q}_2 & \hat{Q}_1 & \hat{Q}_0
\end{bmatrix}.
\end{equation}
Denote the first column of $\hat{\mathbf{Q}}^D$ as $\hat{\mathbf{Q}}^D_c=[\hat{Q}_0^T,-\hat{Q}_1^T,-\hat{Q}_2^T,-\hat{Q}_3^T]^T$. 
Paper \cite{Duan2023} proved that solving a dual quaternion linear equations system
$\hat{\mathbf{Q}}\hat{\mathbf{x}}=\hat{\mathbf{y}}$ is equivalent to solving
$\hat{\mathbf{Q}}^D\hat{\mathbf{x}}^D_c=\hat{\mathbf{y}}^D_c$,
which is a dual number linear equations system.
Suppose that $\hat{\mathbf{Q}}^D=Q_1+Q_2\varepsilon$, $\hat{\mathbf{x}}^D_c=\mathbf{x}_1+\mathbf{x}_2\varepsilon$, $\hat{\mathbf{y}}^D_c=\mathbf{y}_1+\mathbf{y}_2\varepsilon$,
then solving $\hat{\mathbf{Q}}^D\hat{\mathbf{x}}^D_c=\hat{\mathbf{y}}^D_c$
is equivalent to solving a real linear equations system
$$\left\{  
\begin{aligned}  
&Q_1\mathbf{x}_1=\mathbf{y}_1, \\  
&Q_1\mathbf{x}_2+Q_2\mathbf{x}_1=\mathbf{y}_2.
\end{aligned}  
\right.$$

The Rayleigh quotient iteration method is organized as follows, see Algorithm \ref{RQI}. It has a local cubic convergence rate for solving the extreme eigenvalues of dual quaternion Hermitian matrices. 
\begin{algorithm}
    \caption{The Rayleigh quotient iteration method for calculating the extreme eigenvalues of dual quaternion Hermitian matrices} \label{RQI}
    \begin{algorithmic}
    \REQUIRE $\hat{\mathbf Q}\in \hat{\mathbb{H}}^n$, initial iteration point $\hat{\mathbf v}^{(0)}\in\hat{\mathbb{Q}}_2^{n\times 1}$, the maximum iteration number $k_{max}$ and the tolerance $\delta$.
    \ENSURE $\hat{\mathbf v}^{(k)}$ and $\hat{\lambda}^{(k-1)}$.
    \FOR {$k = 1$ to $k_{max}$}   
    \STATE Update $\hat{\lambda}^{(k-1)}=(\hat{\mathbf{v}}^{(k-1)})^*\hat{\mathbf{Q}}\hat{\mathbf{v}}^{(k-1)}$.
    \STATE SolveUpdate $(\hat{\mathbf{Q}}-\hat{\lambda}^{(k-1)}\hat{\mathbf{I}})\hat{\mathbf{u}}^{(k)}=
    \hat{\mathbf{v}}^{(k-1)}$.
    \STATE Update $\hat{\mathbf{v}}^{(k)}=
    \frac{\hat{\mathbf{u}}^{(k)}}{\|\hat{\mathbf{u}}^{(k)}\|_2}$.
    \STATE If $\| \hat{\mathbf y}^{(k)}-\hat{\mathbf v}^{(k-1)}\hat{\lambda}^{(k-1)}\| _{2^R}\le \delta \times \| \hat{\mathbf{Q}}\|_{F^R}$, then Stop.
    \ENDFOR
    \end{algorithmic}
\end{algorithm}

\section{Dual Complex Adjoint Matrix and its Applications}\label{section3}
In this section, we introduce the dual complex adjoint matrix and some useful properties. Leveraging this matrix, we delve into the study of eigenvalues of dual quaternion matrices, the Hand-eye calibration problem and the dual quaternion linear equations systems, which leads to the improvement of the Rayleigh quotient iteration method.

\subsection{Dual complex adjoint matrix}

Let $\tilde{\mathbf Q}= Q_1+ Q_2 i+  Q_3 j+ Q_4  k\in \mathbb Q^{m\times n}$, the complex adjoint matrix \cite{Zhang1997} of quaternion matrix $\tilde{\mathbf Q}$ is defined by
\begin{equation}
J(\tilde{\mathbf Q})=\begin{bmatrix}
  P_1 &   P_2\\
-\overline{ P_2}   & \overline{ P_1} 
\end{bmatrix},
\end{equation}
where $P_1=Q_1+Q_2 i$ and $ P_2= Q_3+ Q_4 i$ are complex matrix. Let $A=(a_{ij})\in \mathbb C^{m\times n}$, then $\overline{A}=(\overline{a_{ij}})$, i.e., take the conjugate of every element of complex matrix $A$.

The complex adjoint matrix has many 
important properties in studying the matrix theory of quaternion matrix \cite{Zhang1997}. One of them is shown in Lemma \ref{lemma3.1}.
\begin{lemma}\label{lemma3.1}
Let $\tilde{\mathbf A}= A_1+ A_2 j\in \mathbb Q^{n\times n}$, $\tilde{\mathbf x}=\mathbf x_1+\mathbf x_2  j\in \mathbb Q^{n\times 1}$, $\lambda\in \mathbb C$, where $A_1, A_2\in\mathbb C^{n\times n}$, $\mathbf x_1,\mathbf  x_2\in\mathbb C^{n\times 1}$, then 
\begin{equation}
\tilde{\mathbf A}\tilde{\mathbf x}=\tilde{\mathbf x}\lambda
\end{equation} is equivalent to 
\begin{equation}
\begin{bmatrix}
 A_1 & A_2\\
-\overline{ A_2}   & \overline{ A_1} 
\end{bmatrix}
\begin{bmatrix}
\mathbf x_1\\
-\overline{ \mathbf x_2}
\end{bmatrix}=
\lambda\begin{bmatrix}
\mathbf x_1\\
-\overline{ \mathbf x_2}
\end{bmatrix}
~or~
\begin{bmatrix}
 A_1 & A_2\\
-\overline{ A_2}   & \overline{A_1} 
\end{bmatrix}
\begin{bmatrix}
\mathbf x_2\\
\overline{\mathbf x_1}
\end{bmatrix}=
\overline{\lambda}\begin{bmatrix}
\mathbf x_2\\
\overline{\mathbf x_1}
\end{bmatrix}.
\end{equation}
\end{lemma}

This lemma states that a quaternion matrix of dimension $n \times n$ possesses exactly $n$ complex eigenvalues, refered to as the standard eigenvalues of a quaternion matrix. Furthermore, this lemma transforms studying eigenvalue theory of quaternion matrices into studying eigenvalue theory of complex adjoint matrices, which is a complex matrix, and the eigenvalue theory of complex matrices is already highly developed and well-established.

Drawing inspiration from the complex adjoint matrix, we aspire to define analogous matrices for dual quaternion matrices while retaining the desirable properties of complex adjoint matrices. Therefore, we introduce the dual complex adjoint matrices, and define the following mapping $\mathcal {J}$, which is a bijection from the set of dual quaternion matrices to the set of dual complex adjoint matrices.
\begin{align}
\mathcal{J}:\hat{\mathbb{Q}}^{m\times n} &\to DM(\mathbb{C}^{m\times n}),
\\A_1+A_2 j+( A_3+ A_4 j)\varepsilon &\mapsto \begin{bmatrix}
 A_1 & A_2 \\ 
-\overline{A_2}   & \overline{ A_1} 
\end{bmatrix} 
+\begin{bmatrix}
 A_3 &  A_4 \\ 
-\overline{ A_4}   & \overline{A_3} 
\end{bmatrix} \varepsilon,
\end{align}
where 
\begin{equation}
DM(\mathbb{C}^{m\times n})=
\left.\left\{\begin{bmatrix}
A_1 &  A_2 \\ 
-\overline{ A_2}   & \overline{ A_1} 
\end{bmatrix} 
+\begin{bmatrix}
 A_3 & A_4 \\ 
-\overline{A_4}   & \overline{ A_3} 
\end{bmatrix} \varepsilon \right| A_1, A_2, A_3, A_4\in\mathbb{C}^{m\times n}\right\}
\end{equation}
is the set of dual complex adjoint matrices with dimension $2m \times 2n$. Specifically, $DM(\mathbb{C}^{m \times n})$ is a subset of $\mathbb {DC}^{2m \times 2n}$. We refer to $\mathcal{J}(\hat{\mathbf Q})$ as the dual complex adjoint matrix of the dual quaternion matrix $\hat{\mathbf Q}$.

There are some useful properties about mapping $\mathcal{J}$ as follows.
\begin{lemma}\label{lemma3.2}
Let $\hat{\mathbf P},\hat{\mathbf P}_1\in \hat{\mathbb{Q}}^{m\times k}$, $\hat{\mathbf Q} \in \hat{\mathbb{Q}}^{k\times n}$, $\hat{\mathbf R}\in \hat{\mathbb{Q}}^{n\times n}$, then 
\begin{itemize}
\item[{\rm (i)}]  $\mathcal{J}(\hat{\mathbf O}^{m\times n})=\hat{O}^{2m\times 2n},\mathcal{J}(\hat{\mathbf I}_n)=\hat{I}_{2n}$.
\item[{\rm (ii)}]  $\mathcal{J}(\hat{\mathbf P}\hat{\mathbf Q})=\mathcal{J}(\hat{\mathbf P})\mathcal{J}(\hat{\mathbf Q})$.
\item[{\rm (iii)}]  $\mathcal{J}(\hat{\mathbf P}+\hat{\mathbf P}_1)=\mathcal{J}(\hat{\mathbf P})+\mathcal{J}(\hat{\mathbf P}_1)$.
\item[{\rm (iv)}]  $\mathcal{J}(\hat{\mathbf P}^*)=\mathcal{J}(\hat{\mathbf P})^*$.
\item[{\rm (v)}]  $\mathcal{J}(\hat{\mathbf R})$ is unitary (Hermitian) if and only if $\hat{\mathbf R}$ is unitary (Hermitian).
\item[{\rm (vi)}]  $\mathcal{J}$ is an isomorphism from ring $(\hat{\mathbb{Q}}^{n\times n},+,\cdot)$ to ring $(DM(\mathbb{C}^{n\times n}),+,\cdot)$.
\end{itemize}
\end{lemma}
\begin{proof}
We prove (ii), (iv), and (v).

Firstly, we prove equation (ii).
Suppose that $\hat{\mathbf P}=P_1+P_2 j+(P_3+P_4 j)\varepsilon$, $\hat{\mathbf Q}=Q_1+Q_2 j+( Q_3+ Q_4 j)\varepsilon$, then 
\begin{equation*}
(\mathcal{J}(\hat{\mathbf P})\mathcal{J}(\hat{\mathbf Q}))_{st}=
\begin{bmatrix}
P_1 &  P_2 \\ 
-\overline{ P_2}   & \overline{P_1} 
\end{bmatrix}
\begin{bmatrix}
 Q_1 & Q_2 \\ 
-\overline{ Q_2}   & \overline{Q_1} 
\end{bmatrix}
=\begin{bmatrix}
P_1Q_1-P_2\overline{Q_2} & P_1Q_2+P_2\overline{Q_1} \\ 
-\overline{P_2} Q_1-\overline{P_1} \overline{Q_2}  & -\overline{P_2} Q_2+\overline{P_1} \overline{Q_1} 
\end{bmatrix},
\end{equation*}
\begin{align*} 
(\mathcal{J}(\hat{\mathbf P})\mathcal{J}(\hat{\mathbf Q}))_{\mathcal I}&=
\begin{bmatrix}
P_1 & P_2 \\ 
-\overline{P_2}   & \overline{P_1} 
\end{bmatrix}
\begin{bmatrix}
 Q_3 &  Q_4 \\ 
-\overline{ Q_4}   & \overline{Q_3} 
\end{bmatrix}+\begin{bmatrix}
P_3 &  P_4 \\ 
-\overline{ P_4}   & \overline{ P_3} 
\end{bmatrix}
\begin{bmatrix}
 Q_1 & Q_2 \\ 
-\overline{Q_2}   & \overline{ Q_1} 
\end{bmatrix}
\\&=\begin{bmatrix}
P_1 Q_3-P_2\overline{Q_4}+P_3Q_1-P_4\overline{Q_2} & P_1Q_4+P_2\overline{Q_3}+P_3Q_2+P_4\overline{Q_1} \\ 
-\overline{P_2}Q_3-\overline{P_1} \overline{Q_4}-\overline{P_4}Q_1-\overline{P_3}\overline{Q_2}  & -\overline{P_2}Q_4+\overline{P_1} \overline{Q_3} -\overline{P_4}Q_2+\overline{P_3} \overline{Q_1}  
\end{bmatrix}.
\end{align*}
Since 
\begin{align*} 
(\hat{\mathbf P}\hat{\mathbf Q})_{st}&=
(P_1+P_2 j)(Q_1+Q_2 j)=(P_1Q_1-P_2\overline{Q_2})+(P_1Q_2+P_2\overline{Q_1}) j,
\end{align*}
and
\begin{align*} 
(\hat{\mathbf P}\hat{\mathbf Q})_{\mathcal I}&=
(P_1+P_2 j)(Q_3+Q_4 j)+(P_3+P_4 j)(Q_1+Q_2 j)
\\&=(P_1Q_3-P_2\overline{Q_4}+P_3Q_1-P_4\overline{Q_2})+
(P_1Q_4+P_2\overline{Q_3}+P_3Q_2+P_4\overline{Q_1}) j,
\end{align*}
then we obtain $\mathcal{J}(\hat{\mathbf P}\hat{\mathbf Q})=\mathcal{J}(\hat{\mathbf P})\mathcal{J}(\hat{\mathbf Q})$.

Next, we prove equation (iv). We have
\begin{align*} 
\mathcal{J}(\hat{\mathbf P}^*)&=\mathcal{J}(P_1^*-P_2^T  j+(P_3^*-P_4^T j)\varepsilon)
\\&=
\begin{bmatrix}
P_1^* & -P_2^T \\ 
P_2^*   & P_1^T
\end{bmatrix} 
+\begin{bmatrix}
P_3^* & -P_4^T \\ 
P_4^*   & P_3^T
\end{bmatrix} \varepsilon 
\\&=
\begin{bmatrix}
P_1 & P_2 \\ 
-\overline{P_2}   & \overline{P_1} 
\end{bmatrix} ^*
+\begin{bmatrix}
P_3 & P_4 \\ 
-\overline{P_4}   & \overline{P_3} 
\end{bmatrix}^* \varepsilon
\\&=\mathcal{J}(\hat{\mathbf P})^*.
\end{align*}

Finally, we prove (v). We have
\begin{align*} 
\hat{\mathbf R}~\text{is unitary}&\Longleftrightarrow \hat{\mathbf R}\hat{\mathbf R}^*=\hat{\mathbf R}^*\hat{\mathbf R}=\hat{\mathbf I}_n
\\&\Longleftrightarrow \mathcal{J}(\hat{\mathbf R}\hat{\mathbf R}^*)=\mathcal{J}(\hat{\mathbf R}^*\hat{\mathbf R})=\mathcal{J}(\hat{\mathbf I}_n)
\\&\Longleftrightarrow \mathcal{J}(\hat{\mathbf R})\mathcal{J}(\hat{\mathbf R}^*)=\mathcal{J}(\hat{\mathbf R}^*)\mathcal{J}(\hat{\mathbf R})=\hat{I}_{2n}
\\&\Longleftrightarrow \mathcal{J}(\hat{\mathbf R})\mathcal{J}(\hat{\mathbf R})^*=\mathcal{J}(\hat{\mathbf R})^*\mathcal{J}(\hat{\mathbf R})=\hat{I}_{2n}
\\&\Longleftrightarrow \mathcal{J}(\hat{\mathbf R})~\text{is unitary}.
\end{align*}
\begin{align*} 
\hat{\mathbf R}~\text{is Hermitian}&\Longleftrightarrow \hat{\mathbf R}=\hat{\mathbf R}^*
\\&\Longleftrightarrow \mathcal{J}(\hat{\mathbf R})=\mathcal{J}(\hat{\mathbf R}^*)
\\&\Longleftrightarrow \mathcal{J}(\hat{\mathbf R})=\mathcal{J}(\hat{\mathbf R})^*
\\&\Longleftrightarrow \mathcal{J}(\hat{\mathbf R})~\text{is Hermitian}.
\end{align*}
\end{proof}

\subsection{Standard eigenvalues of dual quaternion matrices}
In this subsection, we utilize dual complex adjoint matrices to study the eigenvalues of dual quaternion matrices. Firstly, we introduce the definition of two dual quaternions being similar.

\begin{definition}
Let $\hat{p},\hat{q}\in \hat{\mathbb Q}$, if there exists an invertible  dual quaternion $ \hat{u}$ such that
\begin{equation}
\hat{p}=\hat{u}^{-1}\hat{q}\hat{u},
\end{equation}
then $\hat{p} $ and $\hat{q}$ are said to be similar and denoted as $\hat{p}\sim\hat{q}$.
\end{definition}

The similarity relationship of dual quaternions is an equivalent relationship, under which a dual quaternion $\hat{p}$ corresponds to an equivalent class $[\hat{p}]=\{\hat{u}^{-1}\hat{q}\hat{u}|
\hat{u}\in\hat{\mathbb{Q}}~\text{is invertible}\}$. It follows from $$\hat{u}^{-1}\hat{q}\hat{u}=\hat{u}^{-1}|\hat{u}|\hat{q}\frac{1}{|\hat{u}|}\hat{u}=(\frac{\hat{u}}{|\hat{u}|})^*\hat{q}\frac{\hat{u}}{|\hat{u}|},$$
that $[\hat{p}]=\{\hat{u}^*\hat{q}\hat{u}|
\hat{u}\in\hat{\mathbb{U}}\}$.

We have the following observation: if $\hat{\lambda}$ is the right eigenvalue of the dual quaternion matrix $\hat{\mathbf Q}\in\hat{\mathbb{Q}}^{n\times n}$, then any element in $[\hat{\lambda}]$ is the right eigenvalue of $\hat{\mathbf Q}$.

Suppose that $\hat{\mathbf v}$ is the right eigenvector of $\hat{\mathbf Q}$ with respect to the right eigenvalue $\hat{\lambda}$, i.e.,
$\hat{\mathbf Q}\hat{\mathbf v}=\hat{\mathbf v}\hat{\lambda}$.
For any invertible dual quaternion $\hat{q}$, it holds, $\hat{\mathbf Q}(\hat{\mathbf v}\hat{q})=(\hat{\mathbf v}\hat{q})(\hat{q}^{-1}\hat{\lambda}\hat{q})$.
Then $\hat{q}^{-1}\hat{\lambda}\hat{q}$ is also the right eigenvalue of $\hat{\mathbf Q}$, and the corresponding right eigenvector is $\hat{\mathbf v}\hat{q}$. Therefore, any element in $[\hat{\lambda}]$ is the right eigenvalue of $\hat{\mathbf Q}$.

Based on the above observation, we aspire to identify a representative element within the equivalence class $[\hat{\lambda}]$. Suppose that $\hat{\lambda}\in \hat{\mathbb Q}$ is the right eigenvalue of $\hat{\mathbf Q}\in\hat{\mathbb{Q}}^{n\times n}$. Next, we prove that there exists a unique $\hat{\mu}=\mu_1+\mu_2 i+(\mu_3+\mu_4 i)\varepsilon\in \mathbb{DC}$, which satisfies $\mu_2>0$ or $\mu_2=0$ and $\mu_4\ge 0$, such that $[\hat{\lambda}]=[\hat{\mu}]$. At this point, we refer to $\hat{\mu}$ as the standard right eigenvalue of the dual quaternion matrix $\hat{\mathbf Q}$.

Firstly, we employ the subsequent lemma to prove that for any $\hat{\lambda}\in \hat{\mathbb Q}$, there exists $\hat{\mu}\in \mathbb {DC}$ such that $[\hat{\lambda}]=[\hat{\mu}]$.

\begin{lemma}\label{lemma3.3}
Let $\hat{a}\in \hat{\mathbb Q}$, there exist $\hat{q}\in \hat{\mathbb{U}}$ and $\hat{\lambda}\in \mathbb {DC}$, such that $\hat{q}^*\hat{a}\hat{q}=\hat{\lambda}$.
\end{lemma}
\begin{proof}
Suppose that $\hat{a}=a_1+a_2 j+(a_3+a_4 j)\varepsilon$. 
Since the mapping $\mathcal{J}$ is a bijection, it follows from (ii) and (iv) in Lemma \ref{lemma3.2} that we need only demonstrate the existence of a unit dual quaternion $\hat{q}$ and a dual complex number $\hat{\lambda}$ such that
\begin{equation}
\mathcal{J}(\hat{q})^*\mathcal{J}(\hat{a})\mathcal{J}(\hat{q})=\mathcal{J}(\hat{\lambda}).
\end{equation}

Suppose that $\mathcal{J}(\hat{a})=A_1+A_2\varepsilon$, then there exist an eigenvalue $\lambda_1\in \mathbb C$ of 
$A_1$ and $\mathbf x\in \mathbb {C}^{2\times 1}$ is the corresponding unit eigenvector, i.e., $A_1\mathbf x=\lambda _1\mathbf x$. Suppose that $\mathbf x=[x_1 ~ -\overline{x_2}]^T$. It follows from
$A_1\mathbf x=\lambda _1\mathbf x$ that 
\begin{equation*}
a_1x_1-a_2\overline{x_2}=\lambda_1 x_1,
~
-\overline{a_2}x_1-\overline{a_1}\overline{x_2}=-\lambda _1 \overline{x_2}. 
\end{equation*}
Hence $$a_1x_2+a_2\overline{x_1}=\overline{\lambda _1} x_2~,~-\overline{a_2}x_2+\overline{a_1} \overline{x_1}=\overline{\lambda _1} \overline{x_1}.$$
Denote $\mathbf y=[x_2 ~ \overline{x_1}]^T$, then $A_1\mathbf y=\overline{\lambda _1}\mathbf y$. Furthermore, it holds $\mathbf x^*\mathbf y=0$. 
Thus, denote $X=\begin{bmatrix}
x_1  &  x_2 \\
-\overline{x_2}  & \overline{x_1}
\end{bmatrix}$, then $X$ is a unitary matrix and $X^*A_1X=\operatorname{diag}(\lambda_1,\overline{\lambda _1})$.

We consider the first case that $\lambda_1\notin \mathbb R$. 
Since $\lambda_1\notin \mathbb R$, then $\lambda_1-\overline{\lambda _1}\ne 0$.
Suppose that $X^*A_2X=\begin{bmatrix}
\lambda _2  &  z \\
-\overline{z}  & \overline{\lambda _2}
\end{bmatrix}$. Denote $Z=\begin{bmatrix}
0 &  -\frac{z}{\lambda_1-\overline{\lambda _1}}  \\
-\frac{\overline{z}}{\lambda_1-\overline{\lambda _1}} & 0
\end{bmatrix}$. By direct calculation, it holds
\begin{equation*}
(X+XZ\varepsilon )^*( A_1+A_2\varepsilon)(X+XZ\varepsilon )=\operatorname{diag}(\hat{\lambda}, \overline{\hat{\lambda}}),
\end{equation*} 
where $\hat{\lambda}=\lambda_1+\lambda_2\varepsilon$.
Furthermore, since 
\begin{equation*}
(X+XZ\varepsilon )^*(X+XZ\varepsilon )=I+(Z^*+Z)\varepsilon=I+ O\varepsilon=\hat{I},
\end{equation*}
then $\hat{Q}=X+XZ\varepsilon\in DM(\mathbb C)$ is a unitary matix. Denote $\hat{q}=\mathcal{J}^{-1}(\hat{Q})$, then $\hat{ q}$ is a unit dual quaternion and it holds $\mathcal{J}(\hat{q})^*\mathcal{J}(\hat{a})\mathcal{J}(\hat{q})=\mathcal{J}(\hat{\lambda})$.

We consider the second case that $\lambda_1\in \mathbb R$. Since $\lambda_1\in \mathbb R$, then $A_1=\lambda_1 XX^*=\lambda_1 I$. Therefore $a_1+a_2 j=\lambda_1\in \mathbb R$. Suppose that $\lambda_2\in \mathbb C$ is an eigenvalue of $A_2$ and $\mathbf y=[y_1 ~ -\overline{y_2}]^T\in \mathbb {C}^{2\times 1}$ is the corresponding unit eigenvector.
Denote $Y=\begin{bmatrix}
y_1  &  y_2 \\
-\overline{y_2}  & \overline{y_1}
\end{bmatrix}$, similar to the proof from the first case, we can conclude that $Y$ is a unitary matrix and $Y^*A_2Y=\operatorname{diag}(\lambda_2,\overline{\lambda_2})$.
Then it holds
$$Y^*(A_1+A_2\varepsilon)Y=Y^* A_1 Y+\operatorname{diag}(\lambda_2,\overline{\lambda_2})\varepsilon=\lambda_1I+\operatorname{diag}(\lambda_2,\overline{\lambda_2})\varepsilon=\operatorname{diag}(\hat{\lambda},\overline{\hat{\lambda}}),$$ where $\hat{\lambda}=\lambda_1+\lambda_2\varepsilon$. Denote
$\hat{q}=\mathcal{J}^{-1}(Y)=y_1+y_2 j$, then it holds
$\mathcal{J}(\hat{q})^*\mathcal{J}(\hat{a})\mathcal{J}(\hat{q})=\mathcal{J}(\hat{\lambda})$.
\end{proof}
It follows from Lemma \ref{lemma3.3} that there exist a unit dual quaternion $\hat{q}\in \hat{\mathbb U}$ such that $\hat{\mu}=\hat{q}^*\hat{\lambda}\hat{q}$ is a dual complex number, hence $[\hat{\lambda}]=[\hat{\mu}]$. Since $j^*\hat{\mu}j=\overline{\hat{\mu}}$, then $[\hat{\mu}]=[\overline{\hat{\mu}}]$. Suppose that $\hat{\mu}=\mu_1+\mu_2 i+(\mu_3+\mu_4 i)\varepsilon$. According to the following lemma, if there exist another dual complex number $\hat{\eta}\in \mathbb{DC}$ such that $[\hat{\lambda}]=[\hat{\eta}]$, then $\hat{\eta}=\hat{\mu}$ or $\hat{\eta}=\overline{\hat{\mu}}$. Then if we further assume that $\mu_2>0$ or $\mu_2=0$ and $\mu_4\ge 0$, then the uniqueness of $\hat{\mu}$ is guaranteed.

\begin{lemma}\label{dec}
Let $\hat{a}=a_1+a_2\varepsilon$, $\hat{b}=b_1+b_2\varepsilon\in\mathbb{DC}$, $\hat{q}=q_1+q_2 j+(q_3+q_4 j)\varepsilon\in \hat{\mathbb{U}}$, such that $\hat{q}^*\hat{a}\hat{q}=\hat{b}$, then
$\hat{a}=\hat{b}$ or $\hat{a}=\overline{\hat{b}}$.
\end{lemma}
\begin{proof}
We consider the first case that $a_1\in \mathbb R$.
The equation $\hat{q}^*\hat{a}\hat{q}=\hat{b}$ yields
\begin{align*}
b_1+b_2\varepsilon&=\hat{q}^*a_1\hat{q}+\hat{q}^*(a_2\varepsilon)\hat{q}\\&=a_1\hat{q}^*\hat{q}+(\overline{q_1}-q_2 j)a_2(q_1+q_2 j)\varepsilon\\&=a_1+(a_2\overline{q_1}q_1+\overline{a_2}q_2\overline{q_2}+(a_2-\overline{a_2})\overline{q_1}q_2 j)\varepsilon.
\end{align*}
Then $(a_2-\overline{a_2})\overline{q_1}q_2=0$. This indicates that 
at least one of $a_2\in \mathbb R$, $q_1=0$ and $q_2=0$ holds.
If $a_2\in \mathbb R$ holds, then $b_2=a_2\overline{q_1}q_1+\overline{a_2}q_2\overline{q_2}=a_2(\overline{q_1}q_1+q_2\overline{q_2})=a_2$. If $q_1=0$, then
$b_2=\overline{a_2}q_2\overline{q_2}=\overline{a_2}$. If $q_2=0$, then $b_2=a_2\overline{q_1}q_1=a_2$. Therefore, $b_2=a_2$ or $b_2=\overline{a_2}$ holds. Furthermore, since $a_1$ is a real number, then $\hat{a}=\hat{b}$ or $\hat{a}=\overline{\hat{b}}$ holds.

We consider the second case that $a_1\notin \mathbb R$. Consider the standard parts on both sides of equation $\hat{q}^*\hat{a}\hat{q}=\hat{b}$, it holds 
\begin{align*} 
b_1&=(\overline{q_1}-q_2 j)a_1(q_1+q_2 j)
\\&=a_1\overline{q_1}q_1+\overline{a_1}\overline{q_2}q_2+(a_1-\overline{a_1})\overline{q_1}q_2 j.
\end{align*}
It follows from $b_1\in \mathbb C$ and $a_1\notin \mathbb R$, that $\overline{q_1}q_2=0$. Hence, at least one of $q_1=0$ and $q_2=0$ holds.

If $q_2=0$, then it holds $\overline{q_1}q_1=1$ and $b_1=a_1$.
Consider the dual parts on both sides of equation $\hat{q}^*\hat{a}\hat{q}=\hat{b}$, it holds
\begin{align*} 
b_2&=\overline{q_1}a_1(q_3+q_4 j)+\overline{q_1}a_2q_1+(\overline{q_3}-q_4 j)a_1q_1
\\&=a_2+a_1(\overline{q_1}q_3+\overline{q_3}q_1)+\overline{q_1}q_4(a_1-\overline{a_1}) j.
\end{align*}
Hence $\overline{q_1}q_4(a_1-\overline{a_1})=0$. It follows from $q_1\ne 0$ and $a_1\notin \mathbb R$, that $q_4=0$. This yields that $\hat{q}$ is a unit dual complex number, then $\hat{1}=(\overline{q_1}+\overline{q_3}\varepsilon )(q_1+q_3\varepsilon)=\overline{q_1}q_1+(\overline{q_1}q_3+\overline{q_3}q_1)\varepsilon$.
Hence $\overline{q_1}q_3+\overline{q_3}q_1=0$, then $b_2=a_2$.
The analysis above shows that if $a_1\notin \mathbb R$ and $q_2=0$, then $\hat{a}=\hat{b}$.

Similarly, it can be proven that if $a_1\notin \mathbb R$ and $q_1=0$, then $\hat{a}=\overline{\hat{b}}$.
\end{proof}

It follows from Lemma \ref{lemma3.3} and Lemma \ref{dec}, that when considering the right eigenvalues of dual quaternion matrices, we can focus on the standard right eigenvalues.

Based on the proof process of Lemma \ref{dec}, we can also present the set of all $\hat{q}$ that satisfy the equation $\hat{q}^*\hat{a}\hat{q}=\hat{b}$.

\begin{corollary}\label{corollary8}
Let $\hat{a}=a_1+a_2\varepsilon$, $\hat{b}=b_1+b_2\varepsilon\in\mathbb{DC}$, $a_1,b_1\notin \mathbb R$, $\hat{q}=q_1+q_2 j+(q_3+q_4 j)\varepsilon\in \hat{\mathbb{U}}$ such that $\hat{q}^*\hat{a}\hat{q}=\hat{b}$, then 
\begin{equation}
\hat{a}=\hat{b}~and~\{\hat{q}\in \hat{\mathbb U}|\hat{a}\hat{q}=\hat{q}\hat{b}\}=\{ \hat{q}|\hat{q} \in \mathbb{DC}\cap \hat{\mathbb U}\},   
\end{equation}
or 
\begin{equation}
\hat{a}=\overline{\hat{b}}~and~\{ \hat{q}\in \hat{\mathbb U}|\hat{a}\hat{q}=\hat{q}\hat{b} \}= \{ \hat{q}j|\hat{q} \in \mathbb{DC}\cap \hat{\mathbb U}\}.    
\end{equation}
\end{corollary}
\begin{proof}
It follows from the proof process of Lemma \ref{dec} for the case $a_1\notin \mathbb R$, at least one of $q_1=0$ and $q_2=0$ holds. If $q_2=0$, then $\hat{a}=\hat{b}$ and $q_2=q_4=0$, i.e., $\hat{q}\in \mathbb{DC}$. Besides, for any $\hat{q}\in \mathbb{DC}\cap \hat{\mathbb U}$, we have $\hat{a}\hat{q}=\hat{q}\hat{b}$. That is to say 
$\{ \hat{q}\in \hat{\mathbb U}|\hat{a}\hat{q}=\hat{q}\hat{b} \}=\{ \hat{q}|\hat{q} \in \mathbb{DC}\cap \hat{\mathbb U}\}$.
On the other hand, if $q_1=0$, then $\hat{a}=\overline{\hat{b}}$ and $q_1=q_3=0$, i.e., $\hat{q}j\in \mathbb{DC}$. Similarly, we have $\{\hat{q}\in \hat{\mathbb U}|\hat{a}\hat{q}=\hat{q}\hat{b}\}=\{ \hat{q}j|\hat{q} \in \mathbb{DC}\cap \hat{\mathbb U}\}$.
\end{proof}

\subsection{Application of dual complex adjoint matrix in Hand-Eye calibration problem}
In 1989, Shiu and Ahmad \cite{qc8} and Tsai and Lenz \cite{qc9} formulated the Hand-Eye calibration problem as solving a homogeneous transformation matrix equation:
$$AX=XB,$$
where A and B represent the coordinate transformation relationship between the two movements of the end-efector and between the two movements of the camera, respectively, and X represents the unknown homogeneous transformation matrix from the robot end-efector frame (hand) to the camera frame (eye). 

To allow a simultaneous computation of the transformations from robot world to robot base and from robot tool to robot flange coordinate frames, Zhuang et al. \cite{qc10} constructed another homogeneous transformation equation
$$AX=YB,$$ where $A$ and $B$ represent the transformation matrix from the robot base to the end-efector and from the world base to the camera, respectively, and $X$ and $Y$ represent the unknown homogeneous transformation matrices from the end-effector to the camera and from the robot base to the world coordinate system, respectively. 

Futhermore, there has also other Hand-Eye calibration methods like calibration based on reprojection error and $AXB=YCZ$.

In three-dimensional Euclidean space, the motion of a rigid body is the rotation and translation of coordinates around the spiral axis, then 
a homogeneous transformation matrix $T$ is defined as 
$T=\begin{bmatrix}
R  & \mathbf t\\ 
0  &1
\end{bmatrix},$
where R is a rotation matrix of size $3\times 3$ and $\mathbf t=(t_1,t_2,t_3)^T\in \mathbb R^{3\times 1}$ is a translation vector of size $3\times 1$.
A rotation can be described as rotation around a unit axis $\mathbf n=(n_1,n_2,n_3)^T\in \mathbb R^{3\times 1}$ with an angle $\theta \in[-\pi,\pi]$. Then the 
rotation matrix R can be formulated as 
$$\small{
\left.R=\left[\begin{array}{ccc}n_1^2+(1-n_1^2)\cos\theta&n_1n_2(1-\cos\theta)-n_3\sin\theta&n_1n_3(1-\cos\theta)+n_2\sin\theta\\n_1n_2(1-\cos\theta)+n_3\sin\theta&n_2^2+(1-n_2^2)\cos\theta&n_2n_3(1-\cos\theta)-n_1\sin\theta\\n_1n_3(1-\cos\theta)+n_2\sin\theta&n_2n_3(1-\cos\theta)+n_1\sin\theta&n_3^2+(1-n_3^2)\cos\theta\end{array}\right.\right].}$$

The 3D motion of a rigid body can also be represented by a unit dual quaternion \cite{Daniilidis1999}. Let 
$\tilde{q}_{st}=\cos\left(\frac\theta2\right)+\sin\left(\frac\theta2\right)n_1i+\sin\left(\frac\theta2\right)n_2j+\sin\left(\frac\theta2\right)n_3k,$ and $\tilde{q}_{\mathcal I}=\frac{1}{2}\tilde{t}\tilde{q}_{st},$
where $\tilde{t}=t_1i+t_2j+t_3k$. Then the unit dual quaternion $\hat{q}=\tilde{q}_{st}+\tilde{q}_{\mathcal I}\varepsilon$ and the homogeneous
transformation matrix $T$ denote the same motion.

\subsubsection{AX=XB Hand-Eye calibration problem}
The $AX=XB$ Hand-Eye calibration problem is to solve the following problem:
\begin{equation}
A^{(i)}X=XB^{(i)}
\end{equation}
for $i=1,2,\cdots,n$, where $X$ is unknown transformation matrix from the robot end-efector frame (hand) to the camera frame (eye), $A^{(i)}$ and $B^{(i)}$ represent the coordinate transformation relationship between the two movements of the end-efector and between the two movements of the camera, respectively. Suppose that the transformation matrices $X$, $A^{(i)}$ and $B^{(i)}$ are encoded with the unit dual quaternions $\hat{x}, \hat{a}_i, \hat{b}_i$, respectively, for $i=1,2,\cdots,n$.
Then the $AX=XB$ Hand-Eye calibration problem can be reformulated as the following problem:
\begin{equation}\label{AX=XB}
\hat{a}_i\hat{x}=\hat{x}\hat{b}_i,
\end{equation}
for $i=1,2,\cdots,n$, where $\{\hat{a}_i \}^n_{i=1},\{ \hat{b}_i\}^n_{i=1} \subset \hat{\mathbb{U}}$ are known in advance and $\hat{x}$ is an undetermined unit dual quaternion.

First, we consider the case that $\hat{a},\hat{b}\in \hat{\mathbb{Q}}$ and there exist $\hat{q}\in\hat{\mathbb{U}}$ such that the equation $\hat{a}\hat{q}=\hat{q}\hat{b}$ holds. 
By Lemma \ref{lemma3.3}, there exist $\lambda,\mu\in\mathbb{DC}$, and $\hat{x},\hat{y}\in \hat{\mathbb U}$, such that $\hat{a}=\hat{x}^*\hat{\lambda}\hat{x}$ and $\hat{b}=\hat{y}^*\hat{\mu}\hat{y}$. Since the equation $\hat{a}\hat{q}=\hat{q}\hat{b}$ holds, 
we have $\hat{x}^*\hat{\lambda}\hat{x}
\hat{q}=\hat{q}\hat{y}^*\hat{\mu}\hat{y}$, 
which is equilvalent to $\hat{\lambda}\hat{x}\hat{q}\hat{y}^*=\hat{x}\hat{q}\hat{y}^*\hat{\mu}$. Then it follows from Lemma \ref{dec} that $\hat{\lambda}=\hat{\mu}$ or $\hat{\lambda}=\overline{\hat{\mu}}$. If $\hat{\lambda}=\overline{\hat{\mu}}$, we can just update $\hat{y}$ by $j\hat{y}$, then we have $\hat{\lambda}=\hat{\mu}$. According to Corollary \ref{corollary8}, If the standard parts of $\hat{\lambda},\hat{\mu}$ are not real numbers, which is equivalent to that the standard parts of $\hat{a},\hat{b}$ are not real numbers, $\hat{x}\hat{q}\hat{y}^*$ is a dual unit complex number. Then $\{\hat{q}|\hat{a}\hat{q}=\hat{q}\hat{b}\}=\{\hat{x}^*\hat{\theta}\hat{y} |\hat{\theta}\in \mathbb{DC}\cap \hat{\mathbb U}\}$. To determine the value of $\hat{\theta}$, we need another equation $\hat{c}\hat{q}=\hat{q}\hat{d}$, i.e., we need at least two equations to solve the $AX=XB$ Hand-Eye calibration problem (\ref{AX=XB}).

The following theorem present the solution of $AX=XB$ Hand-Eye calibration problem.
\begin{theorem}\label{l5}
Let $\hat{a},\hat{b},\hat{c},\hat{d} \in \hat{\mathbb{Q}}$ and there exist $\hat{q} \in \hat{\mathbb{U}}$ such that the equations $\hat{a}\hat{q}=\hat{q}\hat{b}$ and $\hat{c}\hat{q}=\hat{q}\hat{d}$ hold. Suppose that $\hat{a}=\hat{x}^*\hat{\lambda}\hat{x}$ and $\hat{b}=\hat{y}^*\hat{\lambda}\hat{y}$, where $\hat{x},\hat{y} \in \hat{\mathbb U}$, $\hat{\lambda}\in \mathbb{DC}$ and the standard part of $\hat{\lambda}$ is not a real number. Suppose that $\hat{x}\hat{c}\hat{x}^*=c_1+c_2j+(c_3+c_4j)\varepsilon$ and $\hat{y}\hat{d}\hat{y}^*=d_1+d_2j+(d_3+d_4j)\varepsilon$.
If $c_2\ne 0$, then $\hat{q}$ has only two solutions, and the two solutions are only one sign different. 
We have
\begin{equation}
\hat{q}=\pm \hat{x}^*(\overline{\sqrt{\frac{d_2}{c_2} } } +
\frac{\overline{c_2d_4}-\overline{d_2c_4}}{2\overline{c_2}^2}\overline{\sqrt{\frac{c_2}{d_2}}}\varepsilon)\hat{y}.
\end{equation}
\end{theorem}
\begin{proof}
According to the analysis above, we have $$\{\hat{q}|\hat{a}\hat{q}=\hat{q}\hat{b}\}=\{\hat{x}^*\hat{\theta}\hat{y} |\hat{\theta}\in \mathbb{DC}\cap \hat{\mathbb U}\}.$$ Then the equation $\hat{\theta}^*\hat{x}\hat{c}\hat{x}^*\hat{\theta}=\hat{y}\hat{d}\hat{y}^*$ holds. Suppose that $\hat{\theta}=\theta_1+\theta_2\varepsilon$. Consider the standard part of the equation $\hat{\theta}^*\hat{x}\hat{c}\hat{x}^*\hat{\theta}=\hat{y}\hat{d}\hat{y}^*$, we have 
$$d_1+d_2j=\overline{\theta_1}(c_1+c_2j)\theta_1
=\overline{\theta_1}\theta_1c_1+\overline{\theta_1}^2c_2j.$$
It follows from $c_2\ne 0$ that $\theta_1=\pm \overline{\sqrt{\frac{d_2}{c_2}}}$.  
Consider the dual part of the equation $\hat{\theta}^*\hat{x}\hat{c}\hat{x}^*\hat{\theta}=\hat{y}\hat{d}\hat{y}^*$, we have 
\begin{align*} 
d_3+d_4j&=\overline{\theta _1}(c_3+c_4j)\theta _1+\overline{\theta _1}(c_1+c_2j)\theta _2+\overline{\theta _2}(c_1+c_2j)\theta _1
\\&=c_3+(\overline{\theta _1}\theta _2+\overline{\theta _2}\theta _1)c_1+(\overline{\theta _1}^2c_4+2\overline{\theta _1}\overline{\theta _2}c_2)j
\\&=c_3+(\overline{\theta _1}^2c_4+2\overline{\theta _1}\overline{\theta _2}c_2)j
\end{align*}
Then $\theta_2=(\overline{d_4}-\theta _1^2\overline{c_4})/(2\theta _1\overline{c_2})$.
Hence we have
\begin{equation*}
\hat{q}=\pm \hat{x}^*(\overline{\sqrt{\frac{d_2}{c_2} } } +
\frac{\overline{c_2d_4}-\overline{d_2c_4}}{2\overline{c_2}^2}\overline{\sqrt{\frac{c_2}{d_2}}}\varepsilon)\hat{y}.
\end{equation*}
\end{proof}

\subsubsection{AX=YB Hand-Eye calibration problem}
The $AX=YB$ Hand-Eye calibration problem is to solve the following problem:
\begin{equation}
A^{(i)}X=YB^{(i)}
\end{equation}
for $i=1,2,\cdots,n$, where $X$ and $Y$ represent the unknown transformation matrices from the end-effector to the camera and from the robot base to the world coordinate system respectively, $A^{(i)}$ and $B^{(i)}$ represent the transformation matrix from the robot base to the end-efector and from the world base to the camera, respectively. Suppose that the transformation matrices $X$,$Y$, $A^{(i)}$ and $B^{(i)}$ are encoded with the unit dual quaternions $\hat{x}, \hat{y}, \hat{a}_i, \hat{b}_i$, respectively, for $i=1,2,\cdots,n$. Then the $AX=YB$ Hand-Eye calibration problem can be reformulated as the following problem:
\begin{equation}
\hat{a}_i\hat{x}=\hat{y}\hat{b}_i,
\end{equation}
for $i=1,2,\cdots,n$, where $\{ \hat{a}_i \}^n_{i=1},\{ \hat{b}_i\}^n_{i=1} \subset \hat{\mathbb{U}}$ are known in advance and $\hat{x},\hat{y}$ are undetermined unit dual quaternions.

We first consider a special case of this problem and find the invariant element.

\begin{lemma}\label{lemma10}
Let $a,b\in \mathbb{C}$ and there exist $\hat{q},\hat{p} \in \hat{\mathbb{U}}$ such that the equation $(1+a\varepsilon)\hat{q}=\hat{p}(1+b\varepsilon)$ holds. Then $\Re(a)=\Re(b)$. 
\end{lemma}
\begin{proof}
Suppose that $\hat{q}=\tilde{q}_1+\tilde{q}_2\varepsilon$ and $\hat{p}=\tilde{p}_1+\tilde{p}_2\varepsilon$. Consider the standard part of the equation $(1+a\varepsilon)\hat{q}=\hat{p}(1+b\varepsilon)$, we have $\tilde{q}_1=\tilde{p}_1$, then consider the dual part, we have 
\begin{equation}\label{aqpb1}
\tilde{q}_2+a\tilde{q}_1=\tilde{q}_1b+\tilde{p}_2.    
\end{equation}
Take conjugate on both sides of the above equation, it holds 
\begin{equation}\label{aqpb2}
\tilde{q}^*_2+\tilde{q}^*_1\overline{a}=\overline{b}\tilde{q}^*_1+\tilde{p}^*_2.
\end{equation}
Left multiply $\tilde{q}^*_1$ on both sides of the equation (\ref{aqpb1}), it holds  
\begin{equation}\label{aqpb3}
\tilde{q}^*_1\tilde{q}_2+\tilde{q}^*_1a\tilde{q}_1=b+\tilde{q}^*_1\tilde{p}_2.    
\end{equation}
Right multiply $\tilde{q}_1$ on both sides of the equation (\ref{aqpb2}), it holds  
\begin{equation}\label{aqpb4}
\tilde{q}^*_2\tilde{q}_1+\tilde{q}^*_1\overline{a}\tilde{q}_1=\overline{b}+\tilde{p}^*_2\tilde{q}_1.
\end{equation}
Since $\hat{q},\hat{p} \in \hat{\mathbb{U}}$, then 
$\tilde{q}^*_1\tilde{q}_2+\tilde{q}^*_2\tilde{q}_1=\tilde{0}$ and $\tilde{q}^*_1\tilde{p}_2+\tilde{p}^*_2\tilde{q}_1=\tilde{0}$.
Then add the equation (\ref{aqpb3}) and equation (\ref{aqpb4}), we have
$$2\Re(b)=b+\overline{b}=\tilde{q}^*_1(a+\overline{a})\tilde{q}_1=2\Re(a)\tilde{q}^*_1\tilde{q}_1=2\Re(a).$$
\end{proof}

Then we consider the general case.

\begin{theorem}\label{AX=YB}
Let $\hat{a},\hat{b}\in \hat{\mathbb{Q}}$ and there exist $\hat{q} ,\hat{p} \in \hat{\mathbb{U}}$ such that the equation $\hat{a}\hat{q}=\hat{p}\hat{b}$ holds. Suppose that $\hat{a}=\hat{x}^*\hat{\lambda}\hat{x}$ and $\hat{b}=\hat{y}^*\hat{\mu}\hat{y}$, where $\hat{x},\hat{y}\in \hat{\mathbb{U}}$, $\hat{\lambda}=\lambda_1+\lambda_2\varepsilon,\hat{\mu}=\mu_1+\mu_2\varepsilon \in \mathbb{DC}$. Assume that $\lambda_1,\mu_1 \notin \mathbb{R}$. 
Denote
\begin{equation}
Q(\hat{m})=\hat{x}^*(\overline{\lambda_1}/| \lambda_1|)(1-\Im(\lambda_2\overline{\lambda_1}/|\lambda_1|^2)i\varepsilon)\hat{m}\hat{y} 
\end{equation}
and 
\begin{equation}
P(\hat{m})=\hat{x}^*\hat{m}(1-\Im(\mu_2\overline{\mu_1}/ | \mu_1|^2)i\varepsilon)(\overline{\mu_1}/|\mu_1|)\hat{y}.    
\end{equation}
Then 
\begin{equation}\label{e32}
\{(\hat{q},\hat{p})|\hat{a}\hat{q}=\hat{p}\hat{b}\}=\{ (Q(\hat{m}),P(\hat{m}))|\hat{m}\in \hat{\mathbb U}\}.
\end{equation}
\end{theorem}
\begin{proof}
Since $$\lambda_1+\lambda_2\varepsilon=(\left | \lambda_1 \right |+\lambda_2
\overline{\lambda_1}/\left | \lambda_1 \right |\varepsilon)(\lambda_1/\left | \lambda_1 \right |)$$ and $$\mu_1+\mu_2\varepsilon=(\mu_1/\left | \mu_1
\right |)(\left | \mu_1 \right |+\mu_2 \overline{\mu_1}/\left | \mu_1 \right |\varepsilon),$$ then
\begin{equation*}
\hat{x}^*(\left | \lambda_1 \right |+\lambda_2
\overline{\lambda_1}/\left | \lambda_1 \right |\varepsilon)(\lambda_1/\left | \lambda_1 \right |) \hat{x}\hat{q}
=\hat{p}\hat{y}^*(\mu_1/\left | \mu_1
\right |)(\left | \mu_1 \right |+\mu_2
\overline{\mu_1}/\left | \mu_1 \right |\varepsilon) \hat{y}.
\end{equation*}
It is equivalent to 
\begin{equation}\label{AXYB}
(\left | \lambda_1 \right |+\lambda_2
\overline{\lambda_1}/\left | \lambda_1 \right |\varepsilon)(\lambda_1/\left | \lambda_1 \right | ) \hat{x}\hat{q}\hat{y}^*
=\hat{x}\hat{p} \hat{y}^*(\mu_1/\left | \mu_1
\right |)(\left | \mu_1 \right |+\mu_2
\overline{\mu_1}/\left | \mu_1 \right |\varepsilon ).
\end{equation}
Consider the standard part of the above equation, we have $\lambda_1 (\hat{x}\hat{q}\hat{y}^*)_{st}=(\hat{x}\hat{p} \hat{y}^*)_{st}\mu_1,$ then $\left | \lambda_1 \right |=\left | \mu_1 \right |$.
Since both $(\lambda_1/\left | \lambda_1 \right | ) \hat{x}\hat{q}\hat{y}^*$ and $\hat{x}\hat{p} \hat{y}^*(\mu_1/\left | \mu_1 \right |)$ are unit dual quaternions, we have $\Re(\lambda_2\overline{\lambda_1})=\Re(\mu_2\overline{\mu_1})$, according to Lemma \ref{lemma10}.

Suppose that
$$\hat{q}=\hat{x}^*(\overline{\lambda_1}/|\lambda_1|)(1-\Im(\lambda_2\overline{\lambda_1}/ |\lambda_1|^2)i\varepsilon)\hat{m}\hat{y}$$ and $$\hat{p}=\hat{x}^*\hat{n}(1-\Im(\mu_2\overline{\mu_1}/|\mu_1 |^2)i\varepsilon)(\overline{\mu_1}/|\mu_1|)\hat{y},$$
where $\hat{m},\hat{n}\in \hat{\mathbb U}$.
Since \begin{equation*}
|\lambda_1|+\Re(\lambda_2\overline{\lambda_1}/|\lambda_1|)\varepsilon=(|\lambda_1|+\lambda_2\overline{\lambda_1}/ | \lambda_1|\varepsilon)(1-\Im(\lambda_2\overline{\lambda_1}/ | \lambda_1 |^2)i\varepsilon)
\end{equation*}
and 
\begin{equation*}
|\mu_1|+\Re(\mu_2\overline{\mu_1}/|\mu_1|)\varepsilon=(1-\Im(\mu_2\overline{\mu_1}/|\mu_1|^2)i\varepsilon)(|\mu_1|+\mu_2\overline{\mu_1}/|\mu_1|\varepsilon),
\end{equation*}
incorporate the expression of $\hat{p}$ and $\hat{q}$ into the equation (\ref{AXYB}), we get $\hat{m} = \hat{n}$. 
Then $\{ (\hat{q},\hat{p})|\hat{a}\hat{q}=\hat{p}\hat{b}\}=
\{ (Q(\hat{m}),P(\hat{m})) | \hat{m}\in \hat{\mathbb U} \}.$
\end{proof}

We consider the case that $\hat{a},\hat{b}\in \hat{\mathbb{Q}}$ and there exist $\hat{q},\hat{p}\in\hat{\mathbb{U}}$ such that the equation $\hat{a}\hat{q}=\hat{p}\hat{b}$ holds. According to Theorem \ref{AX=YB}, to solve $\hat{q}$ and $\hat{p}$ we need to determine $\hat{m}$ in expresion (\ref{e32}). Then to solve this $AX=YB$ Hand-Eye calibration problem, we require more equations. Let $\hat{a}^{(k)},\hat{b}^{(k)} \in \hat{\mathbb{Q}}$, $k=1,2,\cdots,n$. Suppose that it further holds $\hat{a}^{(k)}\hat{q}=\hat{p}\hat{b}^{(k)}$ for $k=1,2,\cdots,n$.

The equation $\hat{a}^{(k)}\hat{q}=\hat{p}\hat{b}^{(k)}$ is equilvalent to  
\begin{align*} 
\hat{a}^{(k)}\hat{x}^*(\overline{\lambda_1}/|\lambda_1|)(1-\Im(\lambda_2\overline{\lambda_1}/|\lambda_1|^2)i\varepsilon)\hat{m}\hat{y}
\\=\hat{x}^*\hat{m}(1-\Im(\mu_2\overline{\mu_1}/|\mu_1|^2)i\varepsilon)(\overline{\mu_1}/|\mu_1|)\hat{y}\hat{b}^{(k)}.
\end{align*}
Let $$\hat{c}^{(k)}=\hat{x}\hat{a}^{(k)}\hat{x}^*(\overline{\lambda_1}/|\lambda_1|)(1-\Im(\lambda_2\overline{\lambda_1}/| \lambda_1|^2)i\varepsilon),$$ and
$$\hat{d}^{(k)}=(1-\Im(\mu_2\overline{\mu_1}/|\mu_1 |^2)i\varepsilon)(\overline{\mu_1}/|\mu_1|)\hat{y}\hat{b}^{(k)}\hat{y}^*.$$
Then to determine the expression of $\hat{m}$, we just need to solve $$\hat{c}^{(k)}\hat{m}=
\hat{m}\hat{d}^{(k)},$$
for $k=1,2,\cdots,n$, which is a $AX=XB$ Hand-Eye calibration problem, and we have already present the solution of $AX=XB$ Hand-Eye calibration problems in the above subsection.

\subsection{Rayleigh quotient iteration method baesd on dual complex adjoint matrix}

In this subsection, we utilize dual complex adjoint matrices to solve dual quaternion linear equations systems, and then we combine the Rayleigh quotient iteration method with dual complex adjoint matrix.

\subsubsection{Dual quaternion linear equations system}

Linear equations system is an important issue in matrix theory. In this part, we utilize dual complex adjoint matrix to study the problem of dual quaternion linear equations systems.

First, we define the mapping $\mathcal{F}$ from the set $\hat{\mathbb{Q}}^{n\times 1}$ to the set $\mathbb{DC}^{2n\times 1}$ as 
\begin{equation}
\mathcal{F}(\mathbf v_{1}+\mathbf v_{2} j+(\mathbf v_{3}+\mathbf v_{4} j))= 
\begin{bmatrix}
    \mathbf v_{1}\\
    -\overline{\mathbf v_{2}} 
    \end{bmatrix}+
    \begin{bmatrix}
    \mathbf v_{3}\\
    -\overline{\mathbf v_{4}} 
    \end{bmatrix}\varepsilon.  
\end{equation}
Mapping $\mathcal{F}$ is a bijection, and its inverse mapping is
\begin{equation}
\mathcal{F}^{-1}\left ( \begin{bmatrix}
    \mathbf u_{1}\\
    \mathbf u_{2}
    \end{bmatrix}+
    \begin{bmatrix}
    \mathbf u_{3}\\
    \mathbf u_{4}
    \end{bmatrix}\varepsilon \right )=
    \mathbf u_{1}-\overline{\mathbf u_{2}} j+
(\mathbf u_{3}-\overline{\mathbf u_{4}} j  ) \varepsilon,
\end{equation}
where $\mathbf u_{1},\mathbf u_{2},\mathbf u_{3},\mathbf u_{4}
\in \mathbb C^{n\times 1}$.
We define the mapping $\mathcal{G}$ from the set $\hat{\mathbb{Q}}^{n\times 1}$ to the set $\mathbb{DC}^{2n\times 1}$ as 
\begin{equation}
\mathcal{G}(\mathbf v_{1}+\mathbf v_{2} j+(\mathbf v_{3}+\mathbf v_{4} j))= 
\begin{bmatrix}
    \mathbf v_{2}\\
    \overline{\mathbf v_{1}} 
    \end{bmatrix}+
    \begin{bmatrix}
    \mathbf v_{4}\\
    \overline{\mathbf v_{3}} 
    \end{bmatrix}\varepsilon.  
\end{equation}
Then $\mathcal{J}(\hat{\mathbf v})=[\mathcal{F}(\hat{\mathbf v})~\mathcal{G}(\hat{\mathbf v})]$, for any $\hat{\mathbf v}\in\hat{\mathbb{Q}}^{n\times 1}$.

\begin{theorem}\label{theorem3.2}
Let $\hat{\mathbf{Q}}\in\hat{ \mathbb{Q}}^{m\times n}$, $\hat{\mathbf v}\in\hat{ \mathbb{Q}}^{n\times 1}$, and $\hat{\mathbf u}\in\hat{\mathbb{Q}}^{m\times 1}$, then
\begin{equation}
\hat{\mathbf{Q}}\hat{\mathbf v}=\hat{\mathbf u}
\end{equation} is equivalent to
\begin{equation}
\mathcal{J}(\hat{\mathbf{Q}})\mathcal{F}(\hat{\mathbf v})=\mathcal{F}(\hat{\mathbf u}).
\end{equation}
\end{theorem}
\begin{proof}
Since the mapping $\mathcal{J}$ is a bijection, then
$\hat{\mathbf{Q}}\hat{\mathbf v}-\hat{\mathbf u}=\hat{\mathbf O}$
is equivalent to $\mathcal{J}(\hat{\mathbf{Q}}\hat{\mathbf v}-\hat{\mathbf u})=\hat{O}$.
It follows from (ii) and (iii) in Lemma \ref{lemma3.2} that
$$\mathcal{J}(\hat{\mathbf{Q}}\hat{\mathbf v}-\hat{\mathbf u})=\mathcal{J}(\hat{\mathbf{Q}})\mathcal{J}(\hat{\mathbf v})-\mathcal{J}(\hat{\mathbf u})=\mathcal{J}(\hat{\mathbf{Q}})[\mathcal{F}(\hat{\mathbf v})~\mathcal{G}(\hat{\mathbf v})]-[\mathcal{F}(\hat{\mathbf u}),\mathcal{G}(\hat{\mathbf u})].$$
Therefore, if
$\hat{\mathbf{Q}}\hat{\mathbf v}=\hat{\mathbf u}$
, then it holds $\mathcal{J}(\hat{\mathbf{Q}})\mathcal{F}(\hat{\mathbf v})=\mathcal{F}(\hat{\mathbf u})$.

On the other hand, if $\mathcal{J}(\hat{\mathbf{Q}})\mathcal{F}(\hat{\mathbf v})=\mathcal{F}(\hat{\mathbf u})$, suppose that
$$\mathcal{J}(\hat{\mathbf{Q}})=\begin{bmatrix}
\hat{P}_1  & \hat{P}_2 \\
-\overline{\hat{P}_2} & \overline{\hat{P}_1} 
\end{bmatrix}, \mathcal{F}(\hat{\mathbf v})=\begin{bmatrix}
\hat{\mathbf v}_1 \\
-\overline{\hat{\mathbf v}_2} 
\end{bmatrix}, \mathcal{F}(\hat{\mathbf u})=\begin{bmatrix}
\hat{\mathbf u}_1 \\
-\overline{\hat{\mathbf u}_2} 
\end{bmatrix},$$ then,
$$\hat{P}_1\hat{\mathbf v}_1-\hat{P}_2\overline{\hat{\mathbf v}_2}=\hat{\mathbf u}_1,
-\overline{\hat{P}_2}\hat{\mathbf v}_1 - \overline{\hat{P}_1}\overline{\hat{\mathbf v}_2}=-\overline{\hat{\mathbf u}_2}.$$
Therefore,
$$\hat{P}_1\hat{\mathbf v}_2+\hat{P}_2\overline{\hat{\mathbf v}_1}=\hat{\mathbf u}_2,
-\overline{\hat{P}_2}\hat{\mathbf v}_2 +\overline{\hat{P}_1}\overline{\hat{\mathbf v}_1}=\overline{\hat{\mathbf u}_1}.$$
Then it holds $\mathcal{J}(\hat{\mathbf{Q}})\mathcal{G}(\hat{\mathbf v})=\mathcal{G}(\hat{\mathbf u})$.
Thus $\mathcal{J}(\hat{\mathbf{Q}})\mathcal{J}(\hat{\mathbf v})=\mathcal{J}(\hat{\mathbf u})$, then $\hat{\mathbf{Q}}\hat{\mathbf v}=\hat{\mathbf u}$.
\end{proof}

Theorem \ref{theorem3.2} facilitates the transformation of solving a dual quaternion linear equations system into the more manageable task of solving a dual complex linear equations system.
Under the assumption of Theorem \ref{theorem3.2}, suppose that $\hat{P}=\mathcal{J}(\hat{\mathbf Q})=P_1+P_2\varepsilon$, $\hat{\mathbf x}=\mathcal{F}(\hat{\mathbf v})=\mathbf x_1+\mathbf x_2\varepsilon$ and $\hat{\mathbf y}=\mathcal{F}(\hat{\mathbf u})=\mathbf y_1+\mathbf y_2\varepsilon$, then solving the dual complex linear equations system $\hat{P}\hat{\mathbf x}=\hat{\mathbf y}$ is equivalent to solving the following linear equations system: 
$$\left\{
\begin{aligned}
&P_1\mathbf x_1=\mathbf y_1,\\
&P_1\mathbf x_2+P_2\mathbf x_1=\mathbf y_2.
\end{aligned}
\right.$$
To solve the above linear equations system, one only needs to first solve $P_1\mathbf x_1=\mathbf y_1$ to obtain $\mathbf x_1$, and then solve $P_1\mathbf x_2=(\mathbf y_2-P_2\mathbf x_1)$ to get $\mathbf x_2$.

\subsubsection{Improve Rayleigh quotient iteration method by dual complex adjoint matrix}
Leveraging the properties of the dual complex adjoint matrix as previously analyzed, we integrate it with the Rayleigh quotient iteration method to compute the eigenvalues of a dual quaternion Hermitian matrix.

The Rayleigh quotient iteration method for computing the extreme eigenvalues of a dual quaternion Hermitian matrix requires solving a dual quaternion linear equations system  $(\hat{\mathbf{Q}}-\hat{\lambda}^{(k-1)}\hat{\mathbf{I}})\hat{\mathbf{u}}^{(k)}=
\hat{\mathbf{v}}^{(k-1)}$ at the $k$-th iteration step.
Let $\hat{P}=\mathcal{J}(\hat{\mathbf{Q}})$, $\hat{\mathbf x}^{(k)}=\mathcal{F}(\hat{\mathbf v}^{(k)})$, $\hat{\mathbf y}^{(k)}=\mathcal{F}(\hat{\mathbf u}^{(k)})$, then it follow from Theorem \ref{theorem3.2}, it is equivalent to solving
$(\hat{P}-\hat{\lambda}^{(k-1)}\hat{I})\hat{\mathbf{y}}^{(k)}=
\hat{\mathbf{x}}^{(k-1)}$. Suppose that $\hat{P}=P_{st}+P_{\mathcal I}\varepsilon$, $\hat{\mathbf x}^{(k)}=\mathbf x^{(k)}_{st}+\mathbf x^{(k)}_{\mathcal I}\varepsilon$, $\hat{\mathbf y}^{(k)}=\mathbf y^{(k)}_{st}+\mathbf y^{(k)}_{\mathcal I}\varepsilon$, $\hat{\lambda}^{(k)}=\mathbf \lambda^{(k)}_{st}+\mathbf \lambda^{(k)}_{\mathcal I}\varepsilon$. Then we just need to solve the linear equations system 
$$\left\{
\begin{aligned}
&(P_{st}-\lambda^{(k-1)}_{st}I)\mathbf y^{(k)}_{st}=\mathbf x^{(k-1)}_{st},\\
&(P_{st}-\lambda^{(k-1)}_{st}I)\mathbf y^{(k)}_{\mathcal I}=\mathbf u^{(k-1)}_{\mathcal I}-(P_{\mathcal I}-\lambda^{(k-1)}_{\mathcal I} I)\mathbf y^{(k)}_{st}.
\end{aligned}
\right.$$

Hence, we obtain the Rayleigh quotient iteration method based on the dual complex adjoint matrix (Algorithm \ref{algorithm3.3}).

\begin{algorithm}[ht]
    \caption{The Rayleigh quotient iteration method for calculating the extreme eigenvalues of dual quaternion Hermitian matrices based on dual complex adjoint matrix} \label{algorithm3.3}
    \begin{algorithmic}
    \REQUIRE $\hat{\mathbf{Q}}\in \hat{\mathbb H}^n$, initial iteration point $\hat{\mathbf v}^{(0)}\in \hat{\mathbb Q}^{n\times 1}_2$, the maximum iteration number $k_{max}$ and the tolerance $\delta$. 
    \ENSURE $\hat{\mathbf v}^{(k)}$ and $\hat{\lambda}^{(k-1)}$.
    \STATE Compute $\hat{P}=\mathcal{J}(\hat{\mathbf{Q}})$ and
    $\hat{\mathbf x}^{(0)}=\mathcal{F}(\hat{\mathbf v}^{(0)})$.
    \FOR {$k = 1$ to $k_{max}$} 
    \STATE Update $\hat{\lambda}^{(k-1)}=(\hat{\mathbf{v}}^{(k-1)})^*\hat{\mathbf{Q}}\hat{\mathbf{v}}^{(k-1)}$.
    \STATE Solve $(P_{st}-\lambda^{(k-1)}_{st}I)\mathbf y^{(k)}_{st}=\mathbf x^{(k-1)}_{st}$.
    \STATE Solve $(P_{st}-\lambda^{(k-1)}_{st}I)\mathbf y^{(k)}_{\mathcal I}=\mathbf u^{(k-1)}_{\mathcal I}-(P_{\mathcal I}-\lambda^{(k-1)}_{\mathcal I} I)\mathbf y^{(k)}_{st}$.
    \STATE Update $\hat{\mathbf x}^{(k)}=\frac{\hat{\mathbf y}^{(k)}}{\left \| \hat{\mathbf y}^{(k)}\right \|_2 }$.
    \STATE Compute $\hat{\mathbf v}^{(k)}= \mathcal{F}^{-1}(\hat{\mathbf x}^{(k)})$.
    \STATE If $\| \hat{\mathbf y}^{(k)} -\hat{\mathbf x}^{(k-1)}\hat{\lambda}^{(k-1)}\| _{2^R}\le \delta \times \| \hat{\mathbf{Q}}\| _{F^R}$, then Stop.
    \ENDFOR
    \end{algorithmic}
\end{algorithm}

Note that the Algorithm \ref{algorithm3.3} only modifies the solution process for the dual quaternion linear equations system within the algorithm, compared to the original algorithm (Algorithm \ref{RQI}), thereby ensuring that the convergence properties of the Rayleigh quotient iteration method are preserved.

In the following analysis, we delve into the enhancement in the number of floating-point operations achieved by fusing dual complex adjoint matrices. The variant of Algorithm \ref{algorithm3.3} differs solely in the approach to solving the dual quaternion linear equations system, as compared to the original Rayleigh quotient iteration method. Consequently, our focus shifts to examining the computational disparities between these two algorithms in addressing a single dual quaternion linear equations system.

\begin{theorem}\label{R1}
Let $\hat{\mathbf Q}\in \hat{\mathbb{H}}^n$, then in every iteration, Rayleigh quotient iteration method requires $\frac{128}{3}n^3+O(n^2)$ floating-point calculations to solve a dual quaternion linear equations system, while Algorithm \ref{algorithm3.3} requires $\frac{64}{3}n^3+O(n^2)$ floating-point calculations.
\end{theorem}
\begin{proof}
By Algorithm\ref{RQI}, in every iteration, the original Rayleigh quotient iteration method requires to solve two real linear equations systems with a dimension of $4n \times 4n$. By Algorithm\ref{algorithm3.3}, in every iteration, Algorithm \ref{algorithm3.3} requires to solve two complex linear equations systems with a dimension of $2n \times 2n$. Since the coefficient matrices of two linear equations systems are the same, then taking LU decomposition as an example for solving linear equations systems, we just require to perform one LU decomposition in every iteration.

Solving an $n\times n$ linear equations system requires a total of $\frac{1}{6}n(n-1)(2n-1)$ additions, $\frac{1}{6}n(n-1)(2n-1)$ multiplications, and $\frac{1}{2}n(n-1)$ divisions to perform LU decomposition and requires further $n(n-1)$ additions, $n(n-1)$ multiplications and $2n$ divisions to solve linear equations system. Therefore, in every iteration, the original Rayleigh quotient iteration method requires a total of $\frac{64}{3}n^3+24n^2-\frac{22}{3}n$ additions, $\frac{64}{3}n^3+24n^2-\frac{22}{3}n$ multiplications, and $8n^2+14n$ divisions, and in summary, $\frac{128}{3}n^3+O(n^2)$ floating-point calculations.

Considering the arithmetic operations on complex numbers, addition of two complex numbers requires two real additions, multiplication of two complex numbers requires four real multiplications and two real additions, and division of two complex numbers requires six real multiplications, three real additions, and two real divisions. Therefore, in every iteration, Algorithm \ref{algorithm3.3} requires a total of $\frac{32}{3}n^3+30n^2+\frac{19}{3}n$ real additions, $\frac{32}{3}n^3+36n^2+\frac{82}{3}n$ real multiplications, and $4n^2+14n$ real divisions, and in summary, $\frac{64}{3}n^3+O(n^2)$ floating-point calculations.
\end{proof}

By Theorem \ref{R1}, Algorithm \ref{algorithm3.3} roughly requires half the computational cost compared to the original Rayleigh quotient iteration method.

\begin{theorem}\label{R2}
Let $\hat{\mathbf Q}=\tilde{\mathbf Q}_{st}+\tilde{\mathbf Q}_{\mathcal{I}}\in \hat{\mathbb{H}}^n$, if all the eigenvalues of $\tilde{\mathbf Q}_{st}$ are positive, then Algorithm \ref{algorithm3.3} then requires $\frac{32}{3}n^3+O(n^2)$ floating-point calculations to solve a dual quaternion linear equations system, in every iteration.
\end{theorem}
\begin{proof}
Since all the eigenvalues of $\tilde{\mathbf Q}_{st}$ are positive, then according to the structure-preserving property of the mapping $\mathcal J$, which is stated in (v) of Lemma \ref{lemma3.2}, the standard part of $\mathcal J(\hat{\mathbf Q})$ is a positive definite matrix. Consequently, we can utilize the Cholesky decomposition to solve linear equations systems. Since solving an $n\times n$ linear equations system by Cholesky decomposition requires a total of $\frac{1}{6}n(n-1)(n+7)$ additions, $\frac{1}{6}n(n-1)(n+7)$ multiplications, $\frac{1}{2}n(n+3)$ divisions, and $n$ square roots. Therefore, in every iteration, Algorithm \ref{algorithm3.3} requires a total of $\frac{16}{3}n^3+38n^2+\frac{5}{3}n$ real additions, $\frac{16}{3}n^3+44n^2+\frac{74}{3}n$ real multiplications, $4n^2+14n$ real divisions and $2n$ square roots, and in summary, $\frac{32}{3}n^3+O(n^2)$ floating-point calculations.
\end{proof}

By Theorem \ref{R2}, if all the eigenvalues of the standard part of the dual quaternion Hermitian matrix are positive, then Algorithm \ref{algorithm3.3} roughly requires one quarter of the computational cost compared to the original Rayleigh quotient iteration method.

\section{Numerical Experiment} \label{section4}

In this section, we give an example to solve Hand-Eye calibration problem by Lemma \ref{lemma3.3} and Theorem \ref{l5}
and utilize the Rayleigh quotient iteration method based on the dual complex adjoint matrix (Algorithm \ref{algorithm3.3}) to compute the eigenvalues of the Laplacian matrices of graphs in multi-agent formation control.

\subsection{Hand-Eye calibration problem}

First we present an example to verify Lemma \ref{lemma3.3}.
\begin{example}
Let $\hat{a}=1+2i+3j+4k+(4+3i+2j+k)\varepsilon$. Now, we utilize Lemma \ref{lemma3.3} to compute $\hat{\lambda}\in \mathbb{DC}$ and $\hat{q}\in \hat{\mathbb{U}}$ such that $\hat{a}=\hat{q}\hat{\lambda}\hat{q}^*$.

Firstly, We have $\mathcal{J}(a)=A_1+A_2\varepsilon=\begin{bmatrix}
1+2i  & 3+4i \\
-3+4i  & 1-2i
\end{bmatrix}
+\begin{bmatrix}
4+3i  & 2+i \\
-2+i  & 4-3i
\end{bmatrix} \varepsilon.$
Directly calculating the eigendecomposition of matrix $A_1$, let $$X= \begin{bmatrix}0.8281 + 0.0000i & -0.4485 + 0.3364i\\
0.4485 + 0.3364i  & 0.8281 + 0.0000i \end{bmatrix}$$ and $\lambda_1=1.0000 + 5.3852i$, then $A_1=X\operatorname{diag}(\lambda_1,\overline{\lambda_1})X^*$.
By direct computation, we have $$X^*A_2X=\begin{bmatrix}
\lambda _2  &  z \\
-\overline{z}  & \overline{\lambda _2}
\end{bmatrix}= \begin{bmatrix} 4.0000 + 2.9711i & -0.4256 - 2.2341i\\ 0.4256 - 2.2341i &  4.0000 - 2.9711i\end{bmatrix}.$$
Let $$Z=\begin{bmatrix}
0 &  -\frac{z}{\lambda_1-\overline{\lambda _1}}  \\
-\frac{\overline{z}}{\lambda_1-\overline{\lambda _1}} & 0
\end{bmatrix}=\begin{bmatrix}  0.0000 + 0.0000i &  0.2074 - 0.0395i\\ -0.2074 - 0.0395i &   0.0000 + 0.0000i\end{bmatrix}.$$
Then \begin{align*}
\hat{q}&=\mathcal{J}^{-1}(X+XZ)\\&=0.8281 + 0.0000i -0.4485j+ 0.3364k\\&+(4.0000 + 2.9711i  -0.4256j - 2.2341k)\varepsilon.
\end{align*}

By direct computation, we can verify $\hat{a}=\hat{q}\hat{\lambda}\hat{q}^*$, where $\hat{\lambda}=\lambda_1+\lambda_2\varepsilon$.
\end{example}

Next, we present an example of solving a $AX=XB$ Hand-Eye calibration problem by Lemma \ref{lemma3.3} and Theorem \ref{l5}.
\begin{example}
We fix $\hat{q}=\frac{1}{\sqrt{2}}+\frac{1}{\sqrt{2}}j+(\frac{1}{\sqrt{2}}-\frac{1}{\sqrt{2}}j)\varepsilon\in\hat{\mathbb U}$, and we randomly generate $\hat{a},\hat{c}\in\hat{\mathbb U}$. The generated
$\hat{a}$ and $\hat{c}$ are
$$\hat{a}= 0.2168 + 0.4862i -0.7901j - 0.3040k+(-1.1186 - 1.7885i+ 1.6621j + 0.8587k)\varepsilon,$$
$$\hat{c}=-0.4309 - 0.4806i-0.5762j - 0.5014k+(4.0132 + 3.5580i+4.5237j + 4.3305k)\varepsilon.$$

Let $\hat{b}=\hat{q}^*\hat{a}\hat{q}$ and  $\hat{d}=\hat{q}^*\hat{d}\hat{q}$.
Then $$\hat{b}= 0.2168 + 0.3040i -0.7901j + 0.4862k+(-1.1186 + 0.1136i+1.6621j - 2.3966k)\varepsilon,$$
$$\hat{d}=-0.4309 + 0.5014i-0.5762k - 0.4806k+(4.0132 - 5.2916i+4.5237j + 2.5552k)\varepsilon.$$

Now, we solve the $AX=XB$ Hand-Eye calibration problem:
$$\{\hat{a}\hat{q}=\hat{q}\hat{b}, \hat{c}\hat{q}=\hat{q}\hat{d} \}.$$

By Lemma \ref{lemma3.3}, we have $\hat{a}=\hat{x}\hat{\lambda}\hat{x}^*$, and $\hat{b}=\hat{y}\hat{\lambda}\hat{y}^*$
where 
$$\hat{\lambda}= 0.2168 + 0.9762i+(  -1.1186 - 2.5033i)\varepsilon,$$
$$\hat{x}=0.8654+0.1799j - 0.4676k+(-0.1604 - 0.0525i+0.0149j - 0.2911k)\varepsilon,$$
$$\hat{y}=0.8098-0.3075j - 0.4997k+( 0.2825 + 0.3516i+0.6176 j+ 0.0777k)\varepsilon.$$

Let $\hat{e}=\hat{x}^*\hat{c}\hat{x}=e_1+e_2j+(e_3+e_4j)\varepsilon$ and $\hat{f}=\hat{y}^*\hat{d}\hat{y}=f_1+f_2j+(f_3+f_4j)\varepsilon.$
Then
$$\hat{e}=-0.4309 + 0.3831i-0.6288j - 0.5217k+(4.0132 - 2.7970i+4.9544 j+ 4.4102k)\varepsilon,$$
$$\hat{f}=-0.4309 + 0.3831i-0.0303j - 0.8165k+(4.0132 - 2.7970i-1.3065j + 6.6819k)\varepsilon.$$

By Theorem \ref{l5}, 
\begin{align*}
\hat{q}&=\pm \hat{x}(\overline{\sqrt{\frac{f_2}{e_2}}} +
\frac{\overline{e_2f_4}-\overline{f_2e_4}}{2\overline{e_2}^2}\overline{\sqrt{\frac{e_2}{f_2}}}\varepsilon)\hat{y}^*.
\\&=\pm (\frac{1}{\sqrt{2}}+\frac{1}{\sqrt{2}}j+(\frac{1}{\sqrt{2}}-\frac{1}{\sqrt{2}}j)\varepsilon).
\end{align*}
\end{example}

\subsection{Eigenvalues of Laplacian matrix}
In multi-agent formation control, the eigenvalues of the Laplacian matrix of graphs play an important role in studying the stability of the control model \cite{Qi2023}. Since the Laplacian matrix of a graph is a dual quaternion Hermitian matrix, we utilize the Rayleigh quotient iteration method based on the dual complex adjoint matrix proposed in the previous section to solve the eigenvalue problems of the Laplacian matrices in this section.

Given an undirected graph $G=\left ( V,E\right )$ with $n$ point and a unit dual quaternion vector $\hat{\mathbf q}=(\hat{q}_{i})\in \hat{\mathbb{U}}^{n\times 1}$, then the Laplacian matrix $\hat{\mathbf L}$ for graph $G$ with respect to $\hat{\mathbf q}$ is defined by
\begin{equation*}
    \hat{\mathbf L}=\hat{\mathbf D}-\hat{\mathbf A},
\end{equation*}
where $\hat{\mathbf D}$ is a real diagonal matrix, and the value of its diagonal element is the degree of the corresponding vertex in graph $G$, suppose that $\hat{\mathbf A}=(\hat{a}_{ij})$, then
\begin{equation*}
     \hat{a}_{ij}=
    \begin{cases}
    \hat{q}_{i}^{\ast } \hat{q}_{j}, & \text{if} \quad \left ( i,j\right )\in E ,\\ 
    \hat{0}, & \text{otherwise}.
\end{cases}
\end{equation*}

Given an undirected graph $G$, the sparsity $s$ of graph $G$ is defined as $s=\frac{2|E|}{n^2}$, where $|E|$ denotes the number of elements in the edge set $E$. In numerical experiments, generating a graph with sparsity $s$ involves randomly creating an undirected graph with $\frac{s}{2}n^2$ edges.

We present numerical results for computing extreme eigenvalues and eigenvectors of Laplacian matrices using the Rayleigh quotient iteration (RQI) method and an enhanced version based on the dual complex adjoint matrix (Algorithm \ref{algorithm3.3}). All numerical experiments are conducted in MATLAB (2022a) on a laptop of 8G of memory and Inter Core i5 2.3Ghz CPU. 

\begin{table}
  \centering
  \caption{The numerical results of calculating the extreme eigenpairs of Laplacian matrices with different dimension and sparsity using the RQI method and Algorithm \ref{algorithm3.3}}
  \begin{tabular}{cccccccc}
    \hline
    \multirow{2}{*}{$n$} & \multirow{2}{*}{$s$} & \multicolumn{3}{c}{RQI}& \multicolumn{3}{c}{Algorithm \ref{algorithm3.3}}\\
    \multicolumn{1}{c}{}&\multicolumn{1}{c}{}
    &$e_\lambda$ & $n_{iter}$ & $time (s)$&$e_\lambda$& $n_{iter}$ & $time (s)$ \\
    \hline
    10 & 10\% & 4.08e-7 & 2.67 & 8.07e-3  & 5.39e-7 & 2.58 & 6.38e-4\\
    10 & 20\% & 2.97e-7 & 3.36 &  1.06e-2 & 2.47e-7 & 3.44 & 9.64e-4 \\  
    10 & 30\% & 3.23e-7 &  3.56 & 1.03e-2 & 2.09e-7 & 3.65 & 1.02e-3\\
    10 & 40\% & 3.65e-7 &  4.10 & 1.04e-2 & 3.01e-7 & 4.11 & 1.07e-3\\
    10 & 50\% & 3,23e-7 & 4.30 & 1.09e-2 & 3.16e-7 & 4.22 & 1.11e-3\\
    10 & 60\% & 6.53e-7 & 4.59 & 1.11e-2 & 4.02e-7 & 4.80 & 1.63e-3\\
    100 & 5\% & 4.53e-7 & 1.95 & 6.52e-2 & 9.55e-7 & 2.08 & 1.46e-2\\
    100 & 8\% & 5.06e-7 & 2.18 & 6.72e-2 & 9.93e-7 & 2.08 & 1.50e-2\\
    100 &10\% & 6.35e-7 & 2.25 & 8.34e-2 & 7.70e-7 & 2.27 & 1.81e-2\\
    100 & 15\% & 7.82e-7 &  2.39 & 7.12e-2 & 3.88e-7 & 2.49 & 1.60e-2\\
    100 &  18\% & 9.88e-7 & 2.38  & 7.75e-2 & 4.22e-7 & 2.55 & 1.64e-2\\
    100 & 20\% & 7.27e-7 & 2.61 &  7.79e-2 & 5.74e-7 & 2.66 & 1.72e-2\\
    \hline
  \end{tabular}
  \label{table2}
\end{table}

We compare the numerical results of computing extreme eigenpairs of Laplacian matrices with varying sparsity and dimensions $n=10$ and $n=100$ using RQI method and Algorithm \ref{algorithm3.3}. The experimental results are shown in Table \ref{table2}. Let $e_\lambda=\|\hat{\mathbf{L}}\hat{\mathbf{u}}-\hat{\lambda}\hat{\mathbf{u}}\|_{2^R}$, where $ \hat{\lambda}$ and $\hat{\mathbf{u}}$ are eigenvalue and the corresponding eigenvector with unit 2-norm of $\hat{\mathbf{L}}$ computed by RQI method or Algorithm \ref{algorithm3.3}. We use $e_\lambda$ to verify the accuracy of the output of these algorithms. Denote $'time(s)'$ as the average elapsed CPU time in seconds for computing the extreme eigenvalues. Denote $n_{iter}$ as the average number of iterations for computing extreme eigenvalues. All results are averaged over 100 trials with different choices of $\hat{q}$ and different $E$. To ensure a suitable initial point, we pre-process with a certain number of iterations of the power method \cite{Cui2023}, enhancing RQI's practical efficiency.

Table \ref{table2} shows that both two methods achieve comparable accuracy for dimension $n=10$ and $n=100$. Notably, Algorithm \ref{algorithm3.3} exhibits significantly improved computational efficiency, requiring $15\%$ and $25\%$ of the average elapsed CPU time of the original RQI method at dimension $n=10$ and $n=100$, respectively. This underscores the effectiveness of incorporating the dual complex adjoint matrix in enhancing RQI's performance.

\section{Final Remarks}\label{section5}
In this paper, we introduce the dual complex adjoint matrix of dual quaternion matrices and delve into its properties. Leveraging this matrix, we define the standard right eigenvalues of dual quaternion matrices and rigorously prove their uniqueness. Furthermore, we exploit the properties of the dual complex adjoint matrix to directly address the Hand-eye calibration problem for both the $AX = XB$ and $AX = YB$ mathematical models. Subsequently, by employing the dual complex adjoint matrix, we transform the task of solving a dual quaternion linear equations system into solving a dual complex linear equations system, thereby improving the Rayleigh quotient iteration method. Our results reveal that this advancement doubles the efficiency of the Rayleigh quotient iteration method, particularly for computing the eigenvalues of dual quaternion Hermitian matrices whose standard parts of eigenvalues are all positive, where the algorithm's efficiency is approximately quadrupled. This underscores the significant applications of dual complex adjoint matrices in the realm of dual quaternion matrix theory, and we anticipate that our research will pave the way for future endeavors in this field.


\begin{thebibliography}{99}

\bibitem{c1}Bultmann S, Li K, Hanebeck U D. Stereo visual SLAM based on unscented dual quaternion filtering[C]//2019 22th International Conference on Information Fusion (FUSION). IEEE, 2019: 1-8.

\bibitem{qc1} Bryson M, Sukkarieh S. Building a Robust Implementation of Bearing‐only Inertial SLAM for a UAV[J]. Journal of Field Robotics, 2007, 24(1‐2): 113-143.

\bibitem{biq}Clifford. Preliminary sketch of biquaternions[J]. Proceedings of the London Mathematical Society, 1871, 1(1): 381-395.

\bibitem{Cadena2016} Cadena C, Carlone L, Carrillo H, et al. Past, present, and future of simultaneous localization and mapping: Toward the robust-perception age[J]. IEEE Transactions on robotics, 2016, 32(6): 1309-1332.

\bibitem{Cheng2016} Cheng J, Kim J, Jiang Z, et al. Dual quaternion-based graphical SLAM[J]. Robotics and Autonomous Systems, 2016, 77: 15-24.

\bibitem{Cui2023} Cui C, Qi L. A power method for computing the dominant eigenvalue of a dual quaternion Hermitian matrix[J]. Journal of Scientific Computing, 2024, 100(1): 21.

\bibitem{Daniilidis1999} Daniilidis K. Hand-eye calibration using dual quaternions[J]. The International Journal of Robotics Research, 1999, 18(3): 286-298.

\bibitem{Duan2023}Duan S Q, Wang Q W, Duan X F. On Rayleigh Quotient Iteration for Dual Quaternion Hermitian Eigenvalue Problem[J]. arXiv preprint arXiv:2310.20290, 2023.

\bibitem{qc6} Enebuse I, Foo M, Ibrahim B S K K, et al. A comparative review of hand-eye calibration techniques for vision guided robots[J]. IEEE Access, 2021, 9: 113143-113155.

\bibitem{Grisetti2010}Grisetti G, Kümmerle R, Stachniss C, et al. A tutorial on graph-based SLAM[J]. IEEE Intelligent Transportation Systems Magazine, 2010, 2(4): 31-43.

\bibitem{qc7} Jiang J, Luo X, Luo Q, et al. An overview of hand-eye calibration[J]. The International Journal of Advanced Manufacturing Technology, 2022, 119(1): 77-97.

\bibitem{d2}Ling C, He H, Qi L. Singular values of dual quaternion matrices and their low-rank approximations[J]. 
Numerical Functional Analysis and Optimization, 2022, 43(12): 1423-1458.

\bibitem{Ling2023} Ling C, He H, Qi L, et al. von Neumann type trace inequality for dual quaternion matrices[J]. arXiv preprint arXiv:2204.09214, 2022.

\bibitem{d3}Ling C, Qi L, Yan H. Minimax principle for eigenvalues of dual quaternion Hermitian matrices and generalized inverses of dual quaternion matrices[J]. Numerical Functional Analysis and Optimization, 2023, 44(13): 1371-1394.

\bibitem{Li2023}Qi L, Luo Z. Eigenvalues and singular values of dual quaternion matrices[J]. arXiv preprint arXiv:2111.12211, 2021.

\bibitem{Qi2022} Qi L, Ling C, Yan H. Dual quaternions and dual quaternion vectors[J]. Communications on Applied Mathematics and Computation, 2022, 4(4): 1494-1508.

\bibitem{Qi2023} Qi L, Wang X, Luo Z. Dual quaternion matrices in multi-agent formation control[J]. arXiv preprint arXiv:2204.01229, 2022.

\bibitem{qc8} Shiu Y C, Ahmad S. Calibration of wrist-mounted robotic sensors by solving homogeneous transform equations of the form AX= XB[J]. 1987.

\bibitem{qc9} Tsai R Y, Lenz R K. A new technique for fully autonomous and efficient 3 d robotics hand/eye calibration[J]. IEEE Transactions on robotics and automation, 1989, 5(3): 345-358.

\bibitem{c4}Torsello A, Rodola E, Albarelli A. Multiview registration via graph diffusion of dual quaternions[C]//CVPR 2011. IEEE, 2011: 2441-2448.

\bibitem{wang2023dual} Wang H, Cui C, Liu X. Dual r-rank decomposition and its applications[J]. Computational and Applied Mathematics, 2023, 42(8): 349.

\bibitem{Wei2024singular} Wei T, Ding W, Wei Y. Singular value decomposition of dual matrices and its application to traveling wave identification in the brain[J]. SIAM Journal on Matrix Analysis and Applications, 2024, 45(1): 634-660.

\bibitem{Wei2013}Wei E, Jin S, Zhang Q, et al. Autonomous navigation of Mars probe using X-ray pulsars: modeling and results[J]. Advances in Space Research, 2013, 51(5): 849-857.

\bibitem{c5}Wang X, Zhu H. On the comparisons of unit dual quaternion and homogeneous transformation matrix[J]. Advances in Applied Clifford Algebras, 2014, 24: 213-229.

\bibitem{Zhang1997} Zhang F. Quaternions and matrices of quaternions[J]. Linear algebra and its applications, 1997, 251: 21-57.

\bibitem{qc10} Zhuang H, Roth Z S, Sudhakar R. Simultaneous robot/world and tool/flange calibration by solving homogeneous transformation equations of the form AX= YB[J]. IEEE Transactions on Robotics and Automation, 1994, 10(4): 549-554.
\end{thebibliography}
\end{document}